\documentclass[a4paper,12pt]{amsart}

\usepackage{amsmath, amssymb}
\usepackage{tikz}
\usepackage{float}

\allowdisplaybreaks

\newtheorem{thm}{Theorem}[section]
\newtheorem{cor}[thm]{Corollary}
\newtheorem{lem}[thm]{Lemma}
\newtheorem{prop}[thm]{Proposition}

\theoremstyle{definition}

\newtheorem*{conv}{Convention}

\newtheorem*{standing}{Standing assumption}

\theoremstyle{remark}
\newtheorem{rmk}[thm]{Remark}
\newtheorem{example}[thm]{Example}

\numberwithin{equation}{section}

\newcommand{\N}{\mathbb{N}}
\newcommand{\Z}{\mathbb{Z}}
\newcommand{\R}{\mathbb{R}}
\newcommand{\RR}{\mathbb{R}}
\newcommand{\OO}{\mathcal{O}}
\newcommand{\TT}{\mathbb{T}}

\def\Tt{\mathcal{T}}
\def\Cc{\mathcal{C}}

\newcommand{\TC}{\mathcal{T}C}
\newcommand{\Aut}{\operatorname{Aut}}
\newcommand{\clsp}{\overline{\text{span}}}
\newcommand{\QE}{\operatorname{QE}}

\newcommand{\lspan}{\operatorname{span}}
\def\clsp{\operatorname{\overline{span}}}

\date{\today}
\thanks{This research was supported by the Marsden Fund of the Royal Society of New Zealand.}

\begin{document}

\author[an Huef]{Astrid an Huef}
\address{Astrid an Huef, Department of Mathematics and Statistics, University of Otago, PO Box 56, Dunedin 9054, New Zealand.}
\email{astrid@maths.otago.ac.nz}

\author[Kang]{Sooran Kang}
\address{Sooran Kang, Department of Mathematics, Sungkyunkwan University, Seobu-ro 2066, Jangan-gu, Suwon, 16419, Republic of Korea.}
\email{sooran@skku.edu}

\author[Raeburn]{Iain Raeburn}
\address{Iain Raeburn, Department of Mathematics and Statistics, University of Otago, PO Box 56, Dunedin 9054, New Zealand.}
\email{iraeburn@maths.otago.ac.nz}

\title[KMS states on algebras of reducible graphs]{KMS states on the operator algebras of reducible higher-rank graphs}

\begin{abstract}
We study the  equilibrium or KMS states of the Toeplitz $C^*$-algebra of a finite higher-rank graph which is reducible.  The Toeplitz algebra carries a gauge action of a higher-dimensional torus, and  a dynamics arises  by choosing an embedding of the real numbers in the torus. Here we use an embedding  which leads to a dynamics which has previously been identified as ``preferred'', and we scale the dynamics so that $1$ is a critical inverse  temperature.  As with $1$-graphs, we study the strongly connected components of the vertices of the graph.  The behaviour of the KMS states depends on both the graphical relationships between the components and the relative size of the spectral radii of the vertex matrices of the components. 
 
We test our theorems on graphs with two  connected components.  We find that our techniques give a complete analysis of the KMS states with inverse temperatures down to a second critical temperature $\beta_c<1$.  
\end{abstract}

\date{October 21, 2016}

\maketitle

\section{Introduction}

There has recently been a great deal of interest in the KMS states on $C^*$-algebras of directed graphs and higher-rank graphs. The subject started with the theorem of Enomoto, Fujii and Watatani which says that a simple Cuntz-Krieger algebra admits a unique KMS state \cite{EFW}. This theorem was subsequently extended to $C^*$-algebras of finite graphs with sources, where the presence of sources gives rise to other KMS states \cite{KW}. In another direction, Exel and Laca made the important observation that the Toeplitz extension of a Cuntz-Krieger algebra has a much more abundant supply of KMS states, and conducted an extensive analysis of the possible phase transitions \cite{EL}. 

Following the detailed analysis of KMS states on Toeplitz algebras in \cite{LR} and \cite{LRR}, which includes an explicit construction of all KMS states above a critical inverse temperature, an analogous construction was carried out in \cite{aHLRS1} for the Toeplitz algebras of finite graphs from \cite{FR}.  When the Cuntz-Krieger quotient  is simple, the simplex of KMS states collapses at the critical inverse temperature to the single state of \cite{EFW}. The analysis of \cite{aHLRS1} was subsequently extended to reducible graphs in \cite{aHLRS4} (see also \cite{aHRabel}), where the authors found subtle interactions between the ideal structure of the algebra and the behaviour of the KMS states at critical inverse temperatures (of which there can be several for reducible graphs). 

The analysis of KMS states on the  algebras of directed graphs was extended to higher-rank graphs in \cite{aHLRS2, aHLRS3}. One big difference in the higher-rank case is the choice of dynamics: for finite graphs, one lifts the gauge action of $\TT$ to an action of $\RR$, and it doesn't much matter how one does this. (Although other interesting dynamics on graph algebras have been studied by other authors \cite{EL, IK, CL, T,C, CT, M}.) For a graph of rank $k$ (or \emph{$k$-graph}), the natural gauge action on the Toeplitz and graph algebras is an action of $\TT^k$; we can lift this to an action of $\R$ by choosing an embedding $x\mapsto e^{irx}$ of $\RR$ in $\TT^k$, but it matters very much which embedding we choose. There is a \emph{preferred dynamics} for which all $k$ directions go critical at the same inverse temperature, and for which the results in \cite{aHLRS2, aHLRS3} are optimal. This is made precise in \cite{aHLRS2}. In the sequel \cite{aHLRS3}, we showed that the KMS states for the preferred dynamics reflect the internal structure of the graph algebra: there is a unique KMS state at the critical inverse temperature if and only if the algebra is simple.

Here we investigate the KMS states for the preferred dynamics on the Toeplitz algebra of  a reducible $k$-graph $\Lambda$, for which the graph algebra is never simple. Above the largest critical inverse temperature $\beta_c$, the states on the Toeplitz algebra are described by the results in \cite{aHLRS2}. So we concentrate here on what happens at the inverse temperature $\beta_c$, and we scale the dynamics so that $\beta_c=1$. As for $1$-graphs in \cite{aHLRS4}, we need to understand the strongly connected components in the vertex set $\Lambda^0$. The behaviour of the KMS states depends on both the graphical relationships between these components and the relationships between the spectral radii of the different components.

We prove two main results describing different ways in which individual components can influence the behaviour of KMS states on the Toeplitz algebra of $\Lambda$. The first focuses on the role of critical components, which are loosely speaking the components where at least one vertex matrix attains its spectral radius.  For such components, we can often disregard other components which feed into them, and thereby replace our given graph with a smaller one (Theorem~\ref{crit}). The second concerns the opposite situation in which a critical component is hereditary, so that no other components feed into it. In this situation we can adapt techniques from \cite[\S4]{aHLRS4} to extend KMS states of the simple Toeplitz algebra of the component graph to KMS states of the Toeplitz algebra of the whole graph (Theorem~\ref{thm:case2-modified}). We find it quite remarkable that this works as well as it does: the crucial observation is that the commutativity of the $k$ different vertex matrices of a $k$-graph has some powerful consequences for their common Perron-Frobenius theory.

We then illustrate our results by applying them to  families of graphs with  only one or two components. When the graph has three components, there are obviously more possibilities, and we encountered technical difficulties.  The crucial problem is that our strategy involves reducing the problem to smaller graphs by removing hereditary subsets, and when we do this, the induced dynamics on the quotient may no longer be the preferred one.  Since new issues arise, we hope to  pursue this elsewhere. 

\subsection*{Contents} We begin with a section summarising our conventions about higher-rank graphs, their strongly connected components and their Toeplitz algebras. We also include a result describing the ideal and quotient of the Toeplitz algebra associated to a hereditary set of vertices (Proposition~\ref{modhered}), and a brief summary of KMS states. Then we discuss the structure of the vertex matrices of a reducible $k$-graph. It is well-known that for a single non-negative matrix one can order the index set so that the matrix is block upper-triangular, with all the irreducible submatrices appearing as diagonal blocks. (See \cite[\S1.2]{Sen}, or \cite[\S2.3]{aHLRS4} for a version using graphical terminology.) We show that because the vertex matrices of a $k$-graph commute, we can simultaneously block upper-triangularise all the vertex matrices  (Proposition~\ref{utstructure}). This will allow us to exploit the common Perron-Frobenius theory of the vertex matrices of irreducible components \cite[Lemma~2.1]{aHLRS2}, \cite[\S3]{aHLRS3}. We first use these decompositions to discuss the preferred dynamics (see \S\ref{subsecpreferred}).

 We prove the first of our main theorems in \S\ref{secmin}. We show that a crucial role is played by strongly connected components where the spectral radii of the vertex matrices are attained. Theorem~\ref{crit} shows how we can often remove a large hereditary set of vertices without affecting the KMS$_1$ states for the preferred dynamics. Then one hopes to be in a situation where the spectral radii are achieved on a hereditary component, and Theorem~\ref{thm:case2-modified} describes the KMS$_1$ states in this situation. The proof of this theorem is long, and involves a careful analysis of the simultaneous Perron-Frobenius theory of the vertex matrices. 
 
We then turn to an analysis of examples. Graphs with a single strongly connected component were studied in \cite{aHLRS2,aHLRS3}, and we briefly summarise the results in \S\ref{sec1cpt}. We then look at graphs with two strongly connected components. We first investigate the graph-theoretic implications of this assumption, and then we analyse the KMS states under some mild hypotheses on the structure of the graph: we assume, for example, that there are no sinks or sources, that the subgraphs associated to the components are coordinatewise irreducible, and that there are paths from one component to the other. Our analysis of graphs with two components in \S\ref{KMS2cpts} is satisfyingly complete. At inverse temperatures $\beta<1$, we have to deal with non-preferred dynamics on quotients, but the results of \cite{aHKR} suffice to cover this.

Our analysis proceeds, as in \cite{aHLRS4} for $1$-graphs, by reducing problems to smaller graphs where our stronger theorems apply. One wrinkle which we have noticed is that, even though we are happy to assume that our $k$-graphs have no sinks or sources, this property is not necessarily preserved when we pass to quotients. In our final section, we discuss how the results of \cite{aHLRS2} need to be adjusted to cover graphs with sources. We find, as in \cite{KW, aHLRS1}, that they can give rise to extra KMS$_\beta$ states at many different inverse temperatures, and that these KMS states often factor through states of the graph algebra $C^*(\Lambda)$ (Proposition~\ref{KMSsources}). We finish with a short appendix  containing some elementary observations about graphs with a single vertex, as in \cite{DY}.

\section{Background}

\subsection{Higher-rank graphs and their algebras}\label{hrgraphs} A higher-rank graph of rank $k$, or $k$-graph, consists of a countable category $\Lambda$ and a functor $d:\Lambda\to \N^k$ satisfying the facorisation property: if $\lambda\in \Lambda^m:=d^{-1}(m)$ and $m=n+p$, then there exist unique $\mu\in \Lambda^n$ and $\nu\in \Lambda^p$ such that $\lambda=\mu\nu$. We write $\Lambda^0$ for the set of objects, which we call vertices, $r,s:\Lambda\to\Lambda^0$ for the codomain and domain maps, and call elements of $\Lambda^n$ paths of degree $n$. We view the vertex set $\Lambda^0$ as a subset $\Lambda$ of by identifying a vertex with the identity morphism at that vertex. 

In general, we use the usual conventions of the subject. For example, for $v,w\in \Lambda^0$ and $n\in \N^k$, $v\Lambda^nw$ denotes the set of paths $\lambda$ of degree $n$ with $r(\lambda)=v$ and $s(\lambda)=w$. All the graphs in this paper are finite in the sense that $\Lambda^n$ is finite for each $n\in \N^k$.  We write $\{e_i:1\leq i\leq k\}$ for the usual basis of $\N^k$. 

We visualise a $k$-graph by drawing its skeleton, which is the coloured directed graph $(\Lambda^0,\Lambda^1:=\bigcup_iv\Lambda^{e_i},r,s)$ in which all the edges of each degree $e_i$ have been coloured with one of $k$ different colours.  We refer to \cite[\S2]{RSY1} and \cite[Chapter~10]{R} for discussions of $k$-graphs and their skeletons, and the precise relationship is described in \cite{HRSW}. We can also view the  skeleton as an algebraic object by replacing it with the vertex matrices $A_i$, which are the $\Lambda^0\times \Lambda^0$ matrices $A_i$ with integer entries \[A_i(v,w)=|v\Lambda^{e_i}w|.\] The factorisation property then implies that the matrices $\{A_i:1\leq i\leq k\}$ pairwise commute. Because the matrices commute, it makes sense to define $A^n:=\prod_{i=1}^kA_{i}^{n_i}$ for $n\in \N^k$, and then the factorisation property implies that $A^n$ has entries $A^n(v,w)=|v\Lambda^nw|$.

When $k=2$, we call elements of $\Lambda^{e_1}$ blue edges and those of $\Lambda^{e_2}$ red edges. When we draw skeletons of $2$-graphs, we use the following convention:

\begin{conv} When we draw 
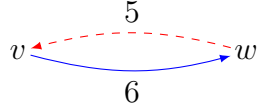
\[
\begin{tikzpicture}[scale=1.5]
 \node[inner sep=0.5pt, circle] (v) at (0,0) {$v$};
    \node[inner sep=0.5pt, circle] (w) at (2,0) {$w$};
\draw[-latex, blue] (v) edge [out=345, in=195]   (w);
\draw[-latex, red, dashed] (w) edge [out=165, in=15]   (v);
\node at (1.0,-.35) {$6$};
\node at (1.0,.35) {$5$}; 
\end{tikzpicture}
\]
in the skeleton of a $2$-graph, we mean that there are $5$ red edges from $w$ to $v$ and $6$ blue edges from $v$ to $w$.
\end{conv}

\subsection{Strongly connected components}\label{subseccomp} Suppose that $\Lambda$ is a $k$-graph. There is a relation $\leq$ on $\Lambda^0$ defined by $v\leq w\Longleftrightarrow v\Lambda w\not=\emptyset$, and this gives an equivalence relation on the vertex set $\Lambda^0$ such that 
\[
v\sim w\Longleftrightarrow v\leq w \text{ and } w\leq v.
\]
We call the equivalence classes \emph{strongly connected components} of $\Lambda$. It is possible that for a vertex $v$ we have $v\Lambda v=\{v\}$, and then $\{v\}$ is an equivalence class. We call such classes \emph{trivial} components, and focus on the set $\Cc$ of nontrivial strongly connected components.

For each $C\in \Cc$, the set $\Lambda_C:=C\Lambda C$ is naturally a $k$-graph. Indeed, what needs to be checked here is that all factorisations of paths in $\Lambda_C$ are themselves in $\Lambda_C$, and there are other subsets of $\Lambda^0$ which have this property: for example, any set $S$ which is \emph{hereditary} ($v\in S$ and $v\leq w$ imply $w\in S$) or \emph{forwards hereditary} ($v\in S$ and $w\leq v$ imply $w\in S$). We denote the vertex matrices of $\Lambda_C$ by $A_{C,i}$. 

Recall that an $n\times n$  matrix $A$ is \emph{irreducible}, if for all $(i,j)$ there exists $m$ such that $A^m(i,j)\neq 0$. 
We will frequently ask that the matrices $\{A_{C,i}:1\leq i\leq k\}$ are all irreducible, in which case the graph $\Lambda_C$ is coordinatewise irreducible in the sense of \cite{aHLRS2}. Since the vertex matrices in any $k$-graph commute,  we then know from \cite[Lemma~2.1]{aHLRS2} that the matrices $\{A_{C,i}\}$ have a common unimodular Perron-Frobenius eigenvector with eigenvalues $\rho(A_{C,i})$ for the matrices $A_{C,i}$.

\subsection{The Toeplitz algebra of a $k$-graph.} Suppose that $\Lambda$ is a finite $k$-graph. For $\mu,\nu\in v\Lambda$, we write
\[
\Lambda^{\min}(\mu,\nu):=\big\{(\eta,\zeta)\in \Lambda\times \Lambda:\mu\eta=\nu\zeta,\ d(\mu\eta)=d(\mu)\vee d(\nu)\big\}.
\]
Following \cite{RS}, a \emph{Toeplitz-Cuntz-Krieger $\Lambda$-family} $\{T,Q\}$ consists
of partial isometries $\{T_\lambda:\lambda\in\Lambda\}$ such that
\begin{itemize}
\item[(T1)] $\{Q_v:=T_v:v\in \Lambda^0\}$ are mutually orthogonal projections;
\item[(T2)] $T_\lambda T_\mu=T_{\lambda\mu}$ whenever $\lambda,\mu\in\Lambda$ with $s(\lambda)=r(\mu)$;
\item[(T3)] $T_\lambda^*T_\lambda=Q_{s(\lambda)}$ for all $\lambda\in\Lambda$;
\item[(T4)] for all $v\in \Lambda^0$ and $n\in\N^k$, we have $Q_v\geq\sum_{\lambda\in v\Lambda^n} T_\lambda T_\lambda^*$;
\item[(T5)] for all $\mu,\nu\in \Lambda$, we have $T_\mu^*T_\nu=\sum_{(\eta,\zeta)\in \Lambda^{\min}(\mu,\nu)} T_\eta T_\zeta^*$.
\end{itemize}
(As usual, we interpret any empty sums in (T4) and (T5) as $0$.) The \emph{Toeplitz algebra} $\TC^*(\Lambda)$ is generated by a universal Toeplitz-Cuntz-Krieger $\Lambda$-family $\{t,q\}$.

\begin{rmk}
For clarity, we have used the same Toeplitz-Cuntz-Krieger relations as \cite{aHLRS2}. In fact they contain redundancy: (T3) follows immediately from (T5) since $\Lambda^{\min}(\lambda,\lambda)=\{s(\lambda)\}$, and, less obviously, (T4) can also be deduced from (T5) (see \cite[Lemma~2.7]{RSY2}).
\end{rmk}

When the graph $\Lambda$ has no sources, the Cuntz-Krieger algebra or graph algebra $C^*(\Lambda)$ is usually taken to be the quotient of $\TC^*(\Lambda)$ in which the inequalities in (T4) become equalities. However, we are going to run into graphs with sources, in which case $v\Lambda^n$ could be empty for some $n$, and then, even though our graphs are finite, we need to use the definition of Cuntz-Krieger family from \cite{RSY2} (see \S\ref{sources} below).

For every $k$-graph, the Toeplitz algebra $\TC^*(\Lambda)$ carries a natural gauge action of $\TT^k$, which is characterised by 
\[
\gamma_z(t_\lambda)=z^{d(\lambda)}t_\lambda:=\big(\textstyle{\prod_{i=1}^k z_i^{d(\lambda)_i}}\big)t_\lambda.
\]
Since each $q_v-\sum_{\lambda\in v\Lambda^n} t_\lambda t_\lambda^*$ is fixed by each $\gamma_z$, this action induces a similar action on the graph algebra $C^*(\Lambda)$, which we also denote by $\gamma$.

\subsection{Ideals and quotients of Toeplitz algebras} Suppose that $\Lambda$ is a $k$-graph. A subset $H$ of $\Lambda^0$ is \emph{hereditary} if $v\in H$ and $v\leq w$ imply $w\in H$. Then the set $\Lambda^0\backslash H$ is forwards hereditary, in the sense that $v\leq w$ and $w\in \Lambda^0\backslash H$ imply $v\in \Lambda^0\backslash H$. Then, as we observed in \S\ref{subseccomp}, the subset $(\Lambda^0\backslash H)\Lambda(\Lambda^0\backslash H)=\Lambda_{\Lambda^0\backslash H}$ is itself a $k$-graph, which we denote by $\Lambda\backslash H$. (Following \cite{RSY1} rather than \cite{S2}, where it was written $\Lambda\backslash \Lambda H$.) The following result is well-known.

\begin{prop}\label{modhered}
Suppose that $\Lambda$ is a $k$-graph, that $H$ is a hereditary subset of $\Lambda^0$, and that $I_H$ is the ideal in $\TC^*(\Lambda)$ generated by $\{q_v:v\in H\}$. 
\begin{enumerate}
\item\label{formIH} We have $I_H=\clsp\{t_\lambda t_\mu^*:s(\lambda)=s(\mu)\in H\}$.
\item\label{hered} There is a homomorphism $q_H$ of  $\TC^*(\Lambda)$ onto 
\[
\TC^*(\Lambda\backslash H)=C^*(p_v,s_\lambda:v,s(\lambda)\in \Lambda^0\backslash H)
\] 
such that $q_H(q_v)=p_v$ for $v\in \Lambda^0\backslash H$ and $q_H(t_\lambda)=s_\lambda$ for $\lambda\in \Lambda\backslash H$.
\item\label{split} There is a homomorphism $\pi$ of $\TC^*(\Lambda \backslash H)=C^*(p,s)$ into $\TC^*(\Lambda)=C^*(q,t)$ such that $\pi(s_\lambda)=t_\lambda$ and $\pi(p_v)=q_v$ for $v\in \Lambda^0\backslash H$ and $\lambda\in \Lambda \backslash H$, and $\pi$ is a splitting for the extension 
\begin{equation*}
0\longrightarrow I_H\longrightarrow \TC^*(\Lambda)\overset{q_H}{\longrightarrow} \TC^*(\Lambda \backslash H)\longrightarrow 0;
\end{equation*}
in other words, $q_H\circ \pi(b)=b$ for $b\in\TC^*(\Lambda \backslash H)$.
\item\label{corneriso} With $P=q|_H$ and $S=t|_{\Lambda H}$, the homomorphism $\pi_{S,P}$ is an isomorphism of $\TC^*(\Lambda_H)$ onto the full corner $p\TC^*(\Lambda)p$ associated to the projection $p=\sum_{v\in H}P_v$ (interpreted as a strict sum if $H$ is infinite, see \cite[Lemma~2.10]{R}).
\end{enumerate}
\end{prop}

\begin{proof}
For \eqref{formIH}, we observe that the set $I:=\clsp\{t_\lambda t_\mu^*:s(\lambda)=s(\mu)\in H\}$ contains the vertex projections $\{q_v:v\in H\}$, and is contained in $I_H$ because each spanning element $t_\lambda t_\mu^*=t_\lambda q_{s(\lambda)} t_\mu^*$. So it suffices to show that $I$ is an ideal. Since $t_\nu t_\lambda=\delta_{s(\nu),r(\lambda)}t_{\nu\lambda}$ and $I$ is $*$-closed, it suffices to see that $t_{\nu}^*t_\lambda t_\mu^*\in I$ when $s(\lambda)=s(\mu)\in H$ and $\nu\in \Lambda$. But then (T5) gives
\[
t_{\nu}^*t_\lambda t_\mu^*=\sum_{(\eta,\zeta)\in\Lambda^{\min}(\nu,\lambda)}t_{\eta}t_{\zeta}^*t_\mu^*=\sum_{(\eta,\zeta)\in\Lambda^{\min}(\nu,\lambda)}t_{\eta}t_{\mu\zeta}^*,
\]
which is in $I$ because $\nu\eta=\lambda\zeta$ and $r(\zeta)=s(\mu)$ implies $s(\zeta)\in H$.

There is a much more general result than \eqref{hered} in \cite[Theorem~4.4]{SWW}, but in this generality the argument of \cite[Proposition~2.1]{aHLRS4} gives a direct proof. For the splitting in \eqref{split}, observe that $Q=q|_{\Lambda^0\backslash H}$, $T=t|_{\Lambda\backslash H}$ is a Toeplitz-Cuntz-Krieger $\Lambda \backslash H$ family in $\TC^*(\Lambda)$. A similar argument gives the homomorphism $\pi_{S,P}$ in \eqref{corneriso}; that it is an isomorphism follows from the uniqueness theorem in \cite[Theorem~8.1]{RS}.
\end{proof}

So for a hereditary set $H$,  $\TC^*(\Lambda \backslash H)$ is both a quotient and a subalgebra of $\TC^*(\Lambda)$. The latter is a result about the Toeplitz algebra rather than the graph algebra: a vertex $v$ in $\Lambda^0\backslash H$ can receive edges from $H$, and then the restriction $(p|_{\Lambda^0\backslash H},s|_{\Lambda\backslash H})$ of the canonical Cuntz-Krieger family in $C^*(\Lambda)$ would not satisfy all the Cuntz-Krieger relations at $v$.

\subsection{KMS states} We are interested in operator-algebraic dynamical systems consisting of an action $\alpha$ of the real numbers $\RR$ on a $C^*$-algebra $A$. An element $a$ of $A$ is analytic for $\alpha$ if the function $t\mapsto \alpha_t(a):\RR\to A$ extends to an analytic function on the complex plane. A state $\phi$ of $A$ is then a KMS$_\beta$ state with inverse temperature $\beta\in (0,\infty)$ if $\phi(ab)=\phi(b\alpha_{i\beta}(a))$ for all $a,b$ in a dense $\alpha$-invariant subalgebra $A_0$ of analytic elements (see \cite[5.3.1]{BR}).

In our examples, the algebra $A$ will be the Toeplitz algebra $\TC^*(\Lambda)$ or graph algebra $C^*(\Lambda)$ of a graph $\Lambda$ of rank $k$. The dynamics will be given in terms of the gauge action $\gamma$ by fixing a vector $r\in \R^k$, and setting 
\[
(\alpha^r)_t=\gamma_{e^{itr}}:=\gamma_{(e^{itr_1},\cdots,e^{itr_k})}. 
\]
Then the subalgebra
$A_0:=\lspan\{t_\mu t_\nu^*:\mu,\nu\in \Lambda\}$ consists of analytic elements and is invariant under $\alpha^r$. So it suffices to check the KMS condition on pairs of elements of the form $t_\mu t_\nu^*$.

\section{Block structure of the vertex matrices}

Suppose that $\Lambda$ is a finite $k$-graph with no sinks or sources, in the strong sense that every vertex receives and emits paths of all degrees. Let $\Cc$ be the collection of nontrivial strongly connected components discussed in \S\ref{subseccomp}. We begin by showing that we can, by carefully ordering the vertex set $\Lambda^0$, give all the vertex matrices $A_i$ a block diagonal form with the vertex matrices $A_{C,i}$ $(C\in\Cc)$ on the diagonal.

\begin{prop}\label{utstructure}
Suppose that $\Lambda$ is a finite $k$-graph without sources or sinks. Let $\Cc$ be the set of strongly connected components of $\Lambda^0$, and suppose that each $\Lambda_C$ for $C\in \Cc$ is coordinatewise irreducible. Then we can order the vertex set $\Lambda^0$ in such a way that each vertex matrix $A_j$ is block upper triangular with the diagonal blocks being the vertex matrices $\{A_{C,j}:C\in \Cc\}$ and some strictly upper triangular matrices.
\end{prop}

\begin{proof}
We first claim that there has to be at least one component $C$ that is forwards hereditary. To see this, suppose not, and fix $C\in \Cc$. There is a vertex $v$ which is not in $C$ such that $v\Lambda C\not=\emptyset$. Since the graph has no sinks, there is an infinite path $x\in \Lambda^{\N e_j}v$ (we deliberately picked an infinite path of a single colour, but it doesn't matter which colour). This path has infinitely many vertices, and hence must contain a return path. The vertices on this return path are all in the same strongly connected component $D_1$, say. Note that $D_1$ cannot be $C$, because then $v$ would have to be in $C$. Since $D_1$ is not forwards hereditary, there exists a vertex $v_1 \notin D_1$ such that $v_1\Lambda D_1\not=\emptyset$.  Continuing in this way yields an infinite sequence $D_n$ of components, none of which is strongly equivalent to any other, which is impossible because $\Lambda$ is finite.

Since there are finitely many components, there are finitely many that are forwards hereditary, say $C_{1,1}, C_{1,2},\dots,C_{1,n_1}$. We list the vertices in $C_{1,1}$ first, then those in $C_{1,2}$, and so on to $C_{1,n_1}$. 

Next we write $C_1:=\bigcup_{i=1}^{n_1}C_{1,i}$, and consider the set
\[
V_1:=\big\{v\in \Lambda^0\backslash C_1:C\in\Cc\text{ and }C\Lambda v\not=\emptyset\Longrightarrow C=C_{1,i}\text{ for some $i\leq n_1$}\big\}.
\]
Following \cite{HRSW}, we say that a path $x=x_1x_2\cdots x_{|x|}$ in the skeleton of $\Lambda$ \emph{traverses} $\lambda\in \Lambda^n$ if each $x_i\in\Lambda^{e_j}$ for some $j$, $|x|=|n|:=n_1+\dots+n_k$ and $\lambda$ is the composition $x_1x_2\cdots x_{|x|}$ in $\Lambda$; we also say that $x$ is a \emph{transversal} of $\lambda$. For $v\in V_1$ we define
\[
i_v=\max\big\{|d(\lambda)|:\lambda\in C_1\Lambda v\text{ has a transversal $x$ such that $s(x_i)\notin C_1$ for all $i$}\big\}.
\]
For $j\leq \max\{i_v:v\in V_1\}$, we set
\[
V_{1,j}:=\{v\in V_1:i_v=j\}.
\]
We order the vertices in $V_1$ by ordering $V_{1,1}$ first, then $V_{1,2}$, and so on.

Now we fix $j$ and look at the matrix $A_j$. Our listing gives us a decomposition $\Lambda^0=C_1\cup V_1\cup R_1$ where $R_1:=\Lambda^0\backslash(C_1\cup V_1)$, and we claim that this gives a corresponding block decomposition
\begin{equation}\label{Ajblock}
A_j=\begin{pmatrix}
A_{C_1,j}&\star&\star\\
0&B_{V_1,j}&\star\\
0&0&A_{R_1,j}
\end{pmatrix}.
\end{equation}
First we observe that the set $C_1$ is forwards hereditary. Thus the bottom two blocks on the left-hand side are $0$.  Next suppose that $w\in V_1$ and $v\in R_1$. Then $v$ connects to at least one component $C$, and since $v$ is not in $V_1$, there exists $C\in\Cc\backslash \{C_{1,i}\}$ such that $C\Lambda v\not=\emptyset$. Then $A_j(v,w)>0$ implies $C\Lambda v\Lambda w\not=\emptyset$, which contradicts $w\in V_1$. So the middle matrix on the bottom row is $0$ too. Thus the matrix $A_j$ has block form \eqref{Ajblock}, as claimed.

Since there are no paths between the disjoint components $C_{1,i}$ and $C_{1,i'}$, the matrix $A_{C_1,j}$ is block diagonal with blocks $A_{C_{1,i},j}$. We now prove that the matrix $B_{V_1,j}$ is strictly upper triangular. Suppose that $v\in V_{1,i}$, $w\in V_{1,i'}$ and $B_{V_1,j}(v,w)>0$. Then there is a path $\mu\in C_1\Lambda v$ with $|d(\mu)|=i$ and a transversal $x$ for $\mu$ such that $s(x_l)\notin C_1$ for all $l$. Since $B_{V_1,j}(v,w)>0$, there is an edge $e\in v\Lambda^{e_j}w$. But then the composition $\mu e$ is a path in $C_1\Lambda w$ with a transversal $y:=xe$ such that $s(y_l)\notin C_1$ for all $l$. Since $|d(\mu e)|=i+1$, we deduce that $i'\geq i+1$. Thus the only nonzero entries in $B_{V_1,j}$ lie above the diagonal. So the matrices
$A_{C_1,j}$ and $B_{V_1,j}$ have the form we require of our diagonal blocks. (Indeed, $A_{C_1,j}$ is block diagonal, which is more than we require.)

We claim that the set $R_1=\Lambda^0\backslash(C_1\cup V_1)$ is hereditary. Suppose that $w\in R_1$ and $v\geq w$. Since $C_1$ is forwards hereditary, $v\in C_1\Longrightarrow w\in C_1$, which is impossible since $w\in R_1$. So $v\notin C_1$. To see that $v$ cannot be in $V_1$, suppose it was, and $C\in\Cc$ has $w\geq C$. Then $v\geq w\geq C$, and $v\in V_1$ implies $C=C_{1,i}$. Thus $w\in V_1$ too. So $v$ cannot be in $V_1$. Thus $v\in  R_1=\Lambda^0\backslash(C_1\cup V_1)$. 

Since $R_1=\Lambda^0\backslash(C_1\cup V_1)$ is hereditary, $\Lambda_{R_1}=R_1\Lambda R_1$ is a $k$-graph. Since $\Lambda$ has no sources and any paths in $\Lambda_{R_1}$ have sources in $R_1$, $\Lambda_{R_1}$ has no sources. To see that it has no sinks, we consider $v\in R_1$. Then $v$ connects to a strongly connected component $C$; it cannot connect only to components of the form $C_{1,i}$, because $v$ is not in $V_1$. So we suppose $C\subset R_1$. Take $\lambda\in C\Lambda v$. Since $\Lambda_C$ is coordinatewise irreducible\footnote{This is where we need coordinatewise irreducible. Without this hypothesis, it is possible that $\Lambda_{R_1}$ has sinks and/or sources. For example, consider the dumbbell graph with $\Lambda^0=\{u,v\}$, one blue loop at each of $u$ and $v$, and one red edge from $v$ to $u$.} $r(\lambda)$ emits edges of all colours with range in $C$. Thus by the factorisation property, so does $v$, and since $\Lambda_C$ is a $k$-graph, these edges must lie in $\Lambda_C$.

Now the graph $\Lambda_{R_1}$ satisfies the hypotheses of the Proposition. Thus we can apply the preceding argument to find forwards hereditary components $\{C_{2,i}:1\leq i\leq n_2\}$ in $\Lambda^0_{R_1}=R_1$, $C_2:=\bigcup_{i=1}^{n_2}C_{2,i}$ and a set $V_2$. This gives us a listing of $R_1$ such that the vertex matrices $A_{R_1,j}$ of $\Lambda_{R_1}$ have the form 
\[
A_{R_1,j}=\begin{pmatrix}
A_{C_2,j}&\star&\star\\
0&B_{V_2,j}&\star\\
0&0&A_{R_2,j}
\end{pmatrix}. 
\]
After finitely many steps, we arrive at a listing of the entire vertex set such that the matrices $A_j$ are simultaneously  block upper triangular with blocks $A_{C_{p,i},j}$ and $B_{V_p,j}$ of the required form. 
\end{proof}

\begin{cor}\label{calcrho}
With $\Lambda$ as in Proposition~\ref{utstructure}, we have
\[
\rho(A_i)=\max\{\rho(A_{C,i}):C\in \Cc\}\quad\text{for $1\leq i\leq k$.}
\]
\end{cor}

\section{The preferred dynamics}\label{subsecpreferred} 

Suppose that $\Lambda$ is a finite $k$-graph without sources or sinks, and $\TC^*(\Lambda)$ is its Toeplitz algebra. For $r\in (0,\infty)^k$, we define a dynamics $\alpha^r:\R\to\Aut \TC^*(\Lambda)$ in terms of the gauge action $\gamma$ of $\TT^k$ by
\[
(\alpha^r)_t=\gamma_{e^{itr}}:=\gamma_{(e^{itr_1},\cdots,e^{itr_k})}. 
\]
When $\Lambda$ is coordinatewise irreducible and there is a KMS$_\beta$ state of $(\TC^*(\Lambda),\alpha^r)$, the inverse temperature $\beta$ satisfies $\beta r_i\geq \ln\rho(A_i)$ for $1\leq i\leq k$ \cite[Corollary~4.3]{aHLRS2}. The number \[\max\{r_i^{-1}\ln\rho(A_i):1\leq i\leq k\}\] is then the \emph{critical inverse temperature}.
For reducible graphs the situation is more complicated, as we shall see, but the next problem is still to find out what happens at the critical inverse temperature $\max\{r_i^{-1}\ln\rho(A_i)\}$. 

We can  scale the dynamics $\alpha^r:t\mapsto \alpha^r_t$ by a factor $c\in (0, \infty)$, so the scaled version sends $t\mapsto \alpha^r_{ct}=\alpha^{c^{-1}r}_t$.  The KMS$_\beta$ states for the scaled dynamics $\alpha^{c^{-1}r}$ are the KMS$_{c^{-1}\beta}$ states for $\alpha^r$ \cite[Lemma~2.1]{aHKR}. So by choosing an appropriate scale factor, we get a dynamics for which the critical inverse temperature has been normalised to $1$. We assume throughout this paper that we have made this normalisation. To sum up:

\begin{standing}
In this paper we study dynamics $\alpha^r:\R\to \Aut \TC^*(\Lambda)$ such that $r\in (0,\infty)^k$ satisfies
\begin{equation}\label{normalise}
\max\big\{r_i^{-1}\ln\rho(A_i):1\leq i\leq k\big\}=1.
\end{equation}
When $r_i=\ln\rho(A_i)$ for all $1\leq i\leq k$, we call $\alpha^r$ the \emph{preferred dynamics}.
\end{standing}

 If $\alpha^r$ is the preferred dynamics, then $r\in (0,\infty)^k$ implies that $\rho(A_i)>1$ for $1\leq i\leq k$; this rules out the case in which some coordinate graph is a disjoint union of cycles (see \cite[Lemma~A.1]{aHLRS1}).
The main results of \cite{aHLRS2, aHLRS3} concern the preferred dynamics, and in \cite{aHKR} we studied more general dynamics satisfying \eqref{normalise}.  
\begin{rmk} 
In this paper, we are primarily interested in  systems $(\TC^*(\Lambda),\alpha^r)$ involving the preferred dynamics. However, the reducible graphs we consider here have nontrivial hereditary sets $H\subset \Lambda^0$, and hence, by Proposition~\ref{modhered}, nontrivial quotients $\TC^*(\Lambda\backslash H)$. The quotient maps all respect the gauge action, and hence $\alpha^r$ induces actions $\bar\alpha^r$ on each quotient. Composing  KMS states of $(\TC^*(\Lambda\backslash H),\bar\alpha^r)$ with the quotient map gives KMS states of the original system, so we want to use our results to find KMS states of the quotient system. But we have to be careful: these quotient dynamics are typically not the preferred dynamics for the graph $\Lambda\backslash H$. However, in  most cases, they will still satisfy the standing assumption \eqref{normalise}. The exception is when we have strict inequality $r_i<\ln\rho(A_{\Lambda\backslash H,i})$ for all $i$, and in that case Theorem~6.1 of \cite{aHLRS2} describes the KMS$_\beta$ states of $(\TC^*(\Lambda\backslash H),\bar\alpha^r)$ for $\beta$ satisfying 
\[
\max\big\{r_i^{-1}\ln\rho(A_{\Lambda\backslash H,i}):1\leq i\leq k\big\}<\beta\leq 1.
\]
\end{rmk}

As we observed above, when $\Lambda$ is coordinatewise irreducible, Corollary~4.3 of \cite{aHLRS2} says that there are no KMS$_\beta$ states unless $\beta\geq 1$.  A more general result is stated in \cite[Proposition~4.1]{aHKR}, but unfortunately that result is not true as stated: if $\Lambda$ is reducible, it is quite possible that there are KMS$_\beta$ states with $\beta<1$ (see \S\ref{KMS2cpts} below). We discuss the flawed proof in \cite{aHKR} in the following remark. 

\begin{rmk}\label{flawinaHKR}
The problem in the proof of \cite[Proposition~4.1]{aHKR} is that (in the notation of that proof) the functional $\phi\circ \pi$ on $\TC^*(E_C)$ need not be a state, and indeed could be $0$. We do not believe that the error affects the rest of that paper, since the result was only used to motivate our preference for the preferred dynamics. 
\end{rmk}

\section{Critical components}\label{secmin}

In this section, we suppose that $\Lambda$ is a finite $k$-graph without sources or sinks, and that $r_i:=\ln\rho(A_i)$ for all $i$, so that $\alpha^r$ is the preferred dynamics on $\TC^*(\Lambda)$. For $j\in \{1,\dots,k\}$ we say that a strongly connected component $C$ of $\Lambda^0$ is \emph{$j$-critical} if $\rho(A_{C,j})=\rho(A_j)$.

\begin{thm}\label{crit}
Suppose that $\Lambda$ is a finite $k$-graph without sources or sinks, and that $\alpha^r$ is the preferred dynamics on $\TC^*(\Lambda)$. Suppose that $C$ is a strongly connected component of $\Lambda^0$, that $\Lambda_C$ is coordinatewise irreducible, and that there exists $j\in\{1,\dots, k\}$ such that $C$ is a  $j$-critical component. Suppose that \[H_j=\{w\in\Lambda^0: C\Lambda^{\N e_j} w\neq\emptyset\}\] is hereditary, and let $H$ be the complement of $C$ in $H_j$. Then every KMS$_1$ state of $(\TC^*(\Lambda),\alpha^r)$ factors through a state of $\TC^*(\Lambda\backslash H)$.
\end{thm}

\begin{proof}
Suppose that $\psi$ is a KMS$_1$ state of $\TC^*(\Lambda)$, and consider the vector $m^\psi=(\,\psi(q_v)\,)$. We take $w\in H$, and aim to prove that $\psi(q_w)=m^\psi_w=0$. Since $w\in H_j$, there exist $v\in C$ and $n\in \N$ such that $v\Lambda^{ne_j}w\not=\emptyset$. Thus $A_j^n(v,w)>0$. 

Applying \cite[Proposition~4.1(a)]{aHLRS2} to the singleton set $K=\{j\}$ and $\beta=1$ shows that $m^\psi$ satisfies
\begin{equation}\label{subinvj}
A_jm^\psi\leq e^{r_j\beta}m^\psi=\rho(A_j)m^\psi.
\end{equation}
We write $D:=\Lambda^0\backslash (C\cup H)$, and write $A_j$ in block form with respect to the decomposition $\Lambda^0=D\cup C\cup H$:
\begin{equation}\label{3block}
A_j=
\begin{pmatrix}A_{D,j}&\star&\star\\0&A_{C,j}&A_{C,H,j}\\0&0&A_{H,j}
\end{pmatrix}.  
\end{equation}

We consider separately the two cases in which $m^\psi|_C=0$ and $m^\psi|_C\not=0$. Suppose first that $m^\psi|_C=0$. The subinvariance relation~\eqref{subinvj} gives
\[
0\leq A_j^n(v,w)m^\psi_w\leq \big(A_j^nm^\psi\big)_v\leq \rho(A_j)^nm^\psi_v=0,
\]
and $A_j^n(v,v)>0$ forces $m^\psi_w=\psi(q_w)=0$.

Now we suppose that $m^\psi|_C$ is nonzero. Then looking at the centre block in \eqref{3block}, the subinvariance inequality $A_jm^\psi\leq  \rho(A_j)m^\psi$ gives
\[
A_{C,j}(m^\psi|_C)\leq (A_jm^\psi)|_C\leq \rho(A_j)m^\psi|_C.
\]
Since $C$ is $j$-critical, we have $\rho(A_j)=\rho(A_{C,j})$ and
\[
A_{C,j}(m^\psi|_C)\leq \rho(A_{C,j})m^\psi|_C,
\]
from which the subinvariance theorem \cite[Theorem~1.6]{Sen} for the irreducible matrix $A_{C,j}$ and the nonzero vector $m^\psi|_C$ implies that 
\begin{equation}\label{PFforj}
A_{C,j}m^\psi|_C=\rho(A_{C,j})m^\psi|_C.
\end{equation}
In other words, $m^\psi|_C$ is a Perron-Frobenius eigenvector for $A_{C,j}$. 
Now
\begin{align*}
\rho(A_{C,j})^nm^\psi_v=\rho(A_j)^nm^\psi_v&\geq(A_j^nm^\psi)_v=\sum_{u\in \Lambda^0}A_j^n(v,u)m^\psi_u\\
&\geq \sum_{u\in C}A_j^n(v,u)m^\psi_u+A_j^n(v,w)m^\psi_w.
\end{align*}
The block decomposition \eqref{3block} implies that $A_j^n(v,u)=A_{C,j}^n(v,u)$ for every $u\in C$, and hence
\begin{align*}
\rho(A_{C,j})^nm^\psi_v
&\geq \sum_{u\in C}A_{C,j}^n(v,u)m^\psi_u+A_j^n(v,w)m^\psi_w\\
&=\rho(A_{C,j})^n\big(m^\psi|_{C}\big)_v+A_j^n(v,w)m^\psi_w\quad\text{by \eqref{PFforj}}\\
&=\rho(A_{C,j})^nm^\psi_v+A_j^n(v,w)m^\psi_w.
\end{align*}
Since $A_j^n(v,w)>0$, we deduce that $\psi(q_w)=m^\psi_w=0$, as desired. Thus $\psi(q_w)=0$ for all $w\in H$.

Since the set $P:=\{q_w:w\in H\}$ consists of projections which are fixed by the dynamics $\alpha$, and since the spanning elements $t_\lambda t_\mu^*$ are analytic with $\alpha_z(t_\lambda t_\mu^*)=e^{iz(d(\lambda)-d(\mu))}t_\lambda t_\mu^*$, it follows from \cite[Lemma~2.2]{aHLRS1} that $\psi$ vanishes on the ideal $I_{H}$ generated by $P$, and hence $\psi$ factors through a KMS$_1$ state of $\Tt C^*(\Lambda)/I_{H}$.  Since $H$ is hereditary, $\Tt C^*(\Lambda)/I_{H}=\Tt C^*(\Lambda\backslash H)$, and the result follows.
\end{proof}

The set $H$ in Theorem~\ref{crit} could be empty, in which case the theorem gives no information.  For example, in the  $2$-graph 
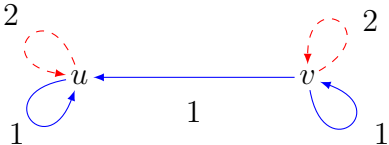
\label{graph-page-ref}
\[\begin{tikzpicture}[scale=1.5]
 \node[inner sep=0.5pt, circle] (u) at (0,0) {$u$};
    \node[inner sep=0.5pt, circle] (v) at (2,0) {$v$};
\draw[-latex, blue] (u) edge [out=190, in=250, loop, min distance=20, looseness=2.5] (u);
\draw[-latex, red, dashed] (u) edge [out=110, in=170, loop, min distance=20, looseness=2.5] (u);
\draw[-latex, blue] (v) edge [out=280, in=340, loop, min distance=20, looseness=2.5] (v);
\draw[-latex, red, dashed] (v) edge [out=30, in=90, loop, min distance=20, looseness=2.5] (v);
\draw[-latex, blue] (v) edge [out=180, in=0]   (u);
\node at (-0.55, -0.5) {$1$};
\node at (-0.6,0.55) {$2$};
\node at (2.65, -0.5) {\color{black} $1$};
\node at (2.55, 0.5 ) {\color{black} $2$};
\node at (1,-0.3) {$1$};
\end{tikzpicture}
\]
the component $C=\{u\}$ is $2$-critical, but there are no red paths from $D=\{v\}$ to $C$, and hence  $H_2=\emptyset$.

On the other hand, the set $H$ in Theorem~\ref{crit} could also be large, and then the theorem does provide useful input.  However, the hereditary set $C\cup H$ need not be large in $\Lambda^0$, and then $\TC^*(\Lambda)$ could have KMS$_\beta$ states for $\beta<1$ lifted from KMS states of the quotient $\TC^*(\Lambda\backslash (C\cup H))$. When $C$ is also minimal among all components, this cannot happen:

\begin{prop}\label{boundbeta} 
Suppose that $\Lambda$ is a finite $k$-graph without sources or sinks, and that $\alpha^r$ is the preferred dynamics on $\TC^*(\Lambda)$. Suppose that $C$ is a strongly connected component of $\Lambda^0$ such that $\Lambda_C$ is coordinatewise irreducible, and that there exists $j$ such that $C$ is $j$-critical. If $C$ has hereditary closure $\Lambda^0$, then every KMS$_\beta$ state of $(\TC^*(\Lambda),\alpha^r)$ has inverse temperature $\beta\geq 1$.
\end{prop}

\begin{proof} 
Suppose that $\psi$ is a KMS$_\beta$ state. We begin by showing that there exists $v\in C$ such that $\psi(q_v)>0$. Since $\sum_{v\in \Lambda^0}q_v$ is the identity of $\TC^*(\Lambda)$, we have $\sum_{v\in \Lambda^0}\psi(q_v)=1$, and there exists $w\in\Lambda^0$ such that $\psi(q_w)>0$. If $w\in C$, take $v=w$. If not, then $w$ belongs to the hereditary closure of $C$, and we have $C\Lambda w\not=\emptyset$. Thus there exists $\lambda\in \Lambda w$ such that $v=r(\lambda)\in C$. Now the Toeplitz-Cuntz-Krieger relation (T4) at $v$, the KMS$_\beta$ condition and (T3) at $w$ give
\[
\psi(q_v)\geq \psi(t_\lambda t_\lambda^*)=e^{-\beta r\cdot d(\lambda)}\psi(t_\lambda^*t_\lambda)=e^{-\beta r\cdot d(\lambda)}\psi(q_w)>0.
\]

 Applying \cite[Proposition~4.1(a)]{aHLRS2} to the singleton set $K=\{j\}$ shows that the vector $m^\psi=(m^\psi_v)=(\,\psi(q_v)\,)$ satisfies
\begin{equation*}
A_jm^\psi\leq e^{r_j\beta}m^\psi=\rho(A_j)^\beta m^\psi.
\end{equation*}
If we order the vertices of $\Lambda$ so that those in $C$ come first, then  $A_j$ has block form
\begin{equation*}
A_j=\begin{pmatrix}A_{C,j}&\star\\0&\star
\end{pmatrix}.
\end{equation*}
Since $C$ is $j$-critical, it follows that $m^\psi|_{C}$ satisfies
\begin{equation*}
A_{C,j}(m^\psi|_{C})\leq \rho(A_{j})^\beta m^\psi|_{C}=\rho(A_{C,j})^\beta m^\psi|_{C}.
\end{equation*}
Since $A_{C,j}$ is irreducible and since the vector $m^\psi|_C$ is non-zero (at the vertex $v$ in the first paragraph), the subinvariance theorem \cite[Theorem~1.6]{Sen} implies that $\rho(A_{C,j})\leq \rho(A_{C,j})^\beta$, and we have $\beta\geq 1$.
\end{proof}

\begin{prop}\label{prop:Cminimal-modified}
Suppose that $\Lambda$ is a finite $k$-graph without sources or sinks, and that $\alpha^r$ is the preferred dynamics on $\TC^*(\Lambda)$. Suppose that $C$ is a strongly connected component of $\Lambda$ such that $\Lambda_C$ is coordinatewise irreducible, and  that there exists $j$ such that  $C$ is $j$-critical and that $H_j:=\{w\in\Lambda^0:C\Lambda^{\N e_j}w\neq\emptyset\}=\Lambda^0$. 
\begin{enumerate}
\item\label{idiK} Every KMS$_1$ state of $(\TC^*(\Lambda),\alpha^r)$ factors through a KMS$_1$ state $\phi$ of the system $(\TC^*(\Lambda_C),\alpha^r)$. Moreover, $\phi$ further factors through the quotient by the ideal generated by
\begin{equation}\label{ideal-gens}
\Big\{q_v-\sum_{e\in v\Lambda_C^{e_i}} t_e t^*_e : v\in \Lambda^0  
\text{and $C$ is $i$-critical} \Big\}.
\end{equation}
\item
Suppose that the numbers $\{\ln \rho(A_{C,i}):1\leq i\leq k\}$ are rationally independent. Then there exists a unique KMS$_1$ state $\psi$ of $(\TC^*(\Lambda),\alpha^r)$, and it satisfies
\begin{equation}\label{notCK}
\psi\Big(q_v-\sum_{f\in v\Lambda_C^{e_i}}t_f t^*_f\Big)\ne 0\quad\text{for $i$ such that $C$ is not $i$-critical.}
\end{equation}
\end{enumerate}
\end{prop}

\begin{proof} Let $\psi$ be a KMS$_1$ state of $(\TC^*(\Lambda),\alpha^r)$.
Since $\{w\in\Lambda^0:C\Lambda^{\N e_j}w\neq\emptyset\}=\Lambda^0$, Theorem~\ref{crit} applies with $H=\Lambda^0\backslash C$, and implies that $\psi$ factors through a state $\phi$ of $\TC^*(\Lambda_C)$. Corollary~\ref{calcrho} implies that $r_i=\ln\rho(A_i)\geq\ln\rho(A_{C,i})$ for all $i$, and since $C$ is $j$-critical, the set $\{i:r_i=\ln\rho(A_{C,i})\}$ is nonempty. So   \cite[Proposition~4.2(b)]{aHKR} applies to $\Lambda_C$, and  implies that $\phi$ factors through the quotient   by the ideal generated by \eqref{ideal-gens}. This gives \eqref{idiK}. 

Since $\rho(A_j)=\rho(A_{C,j})$ and $A_{C,j}$ is irreducible, Proposition~4.2(c) of  \cite{aHKR}  implies that there exists a unique KMS$_1$ state $\phi$ of $(\TC^*(\Lambda_C),\alpha^r)$. Then $\psi=\phi\circ q_{\Lambda^0\backslash C}$ is a KMS$_1$ state of $(\TC^*(\Lambda),\alpha^r)$. Since every KMS$_1$ state factors through $q_{\Lambda^0\backslash C}$, the uniqueness in the first sentence of this paragraph shows that $\psi$ is the only KMS$_1$ state  of $(\TC^*(\Lambda),\alpha^r)$.  Theorem~5.1(c) of \cite{aHKR} gives \eqref{notCK}. 
\end{proof}

\begin{cor}\label{corpreferred-modified}
Suppose that $\Lambda$ is a finite $k$-graph without sources or sinks, and that $\alpha^r$ is the preferred dynamics on $\TC^*(\Lambda)$. Suppose that $C$ is a strongly connected component of $\Lambda$ such that $\Lambda_C$ is coordinatewise irreducible,  that $C$ is $i$-critical for every $i$ and that there exists a $j$ such that $\{w\in\Lambda^0:C\Lambda^{\N e_j}w\neq\emptyset\}=\Lambda^0$. Then
every KMS$_1$ state of $\TC^*(\Lambda)$ factors though a state of $C^*(\Lambda_C)$.
\end{cor}

\begin{proof}
In this case the set at \eqref{ideal-gens} is \[\Big\{q_v-\sum_{e\in v\Lambda_C^{e_i}} t_e t^*_e : v\in \Lambda^0   \Big\},\]
and hence the result follows from part \eqref{idiK}  of Proposition~\ref{prop:Cminimal-modified}.
\end{proof}

\section{When a hereditary component dominates}\label{Dbigger}

In the last section, we discussed graphs in which the spectral radius of one or more vertex matrices are achieved on components which have a nontrivial hereditary closure. We now investigate graphs in which the spectral radii are all achieved on a hereditary component.

We suppose as usual that $\Lambda$ is a finite $k$-graph without sinks or sources.  In the main theorem of this section (Theorem~\ref{thm:case2-modified}), we suppose that there is a strongly connected component $D$ which is also hereditary, and for which the graph $\Lambda_D=D\Lambda D$ is coordinatewise irreducible. Suppose further that 
\begin{equation*}
\rho(A_i)=\rho(A_{D,i})>\rho(A_{C,i})\quad\text{for $1\leq i\leq k$ and $C\in\Cc\backslash\{D\}$ satisfying $C\Lambda D\not=\emptyset$.}
\end{equation*} (For otherwise, some component $C$ with $C\Lambda D\neq\emptyset$ would be $j$-critical, and provided $C\Lambda^{\N e_j} D\neq\emptyset$, we  would apply Theorem~\ref{crit} to replace $\Lambda$ with a smaller graph.) The hypotheses in the next two propositions are a little weaker than this: they do not imply that $\rho(A_{D,i})=\rho(A_i)$ for all $i$. We have included these slightly stronger results because we can see a version of Theorem~\ref{thm:case2-modified} with these weaker hypotheses might be very useful.

For graphs with the properties described above, the preferred dynamics on $\TC^*(\Lambda)$ is also the preferred dynamics on the subalgebra $\TC^*(\Lambda_D)$. There is always a KMS$_1$ state on $\TC^*(\Lambda_D)$, and we will show that this can be scaled back to a KMS functional and then extended to a state on $\TC^*(\Lambda)$. To motivate our construction, we recall what worked for a $1$-graph $E$ in \cite[Theorem~4.3]{aHLRS4}. There we extended the unimodular Perron-Frobenius eigenvector $x$ for $A_D$ to an eigenvector $z=(y,x)$ for $A$ with eigenvalue $\rho(A)$; to do this, we wrote the vertex matrix in block form
\[
A=\begin{pmatrix} A_{E^0\backslash D} & A_{E^0\backslash D,D}\\ 0 & A_{D}\end{pmatrix}
\]
with $\rho(A_{E^0\backslash D})<\rho(A_D)=\rho(A)$, and took
\begin{equation}\label{fromk=1}
y=\big(\rho(A)1_{E^0\backslash D}-A_{E^0\backslash D}\big)^{-1}A_{E^0\backslash D,D}x.
\end{equation}
We then showed by a limiting argument that there is a KMS$_{\ln\rho(A)}$ functional $\phi$ on $\TC^*(E)$ with $(y,x)=\big(\phi(q_v)\big)_{v\in E^0}$. (Had we normalised the dynamics in \cite{aHLRS4} as we have done here, this would have been a KMS$_1$ functional rather than a a KMS$_{\ln\rho(A)}$ functional.) On the face of it, to apply \eqref{fromk=1} to a $k$-graph we  would have to first choose a vertex matrix $A_i$. But, remarkably, it turns out that we always get the same $y$ and $z$. It is crucial in the proof of this that the different vertex matrices commute with each other.

\begin{prop}\label{evectors}
Suppose that $\Lambda$ is a finite $k$-graph without sources or sinks. Suppose that $D$ is a nontrivial strongly connected component which is hereditary, and that for every other component $C\in\Cc$ such that $C\Lambda D\not= \emptyset$, we have 
\begin{equation}\label{Dhered}
\rho(A_{D,i})>\rho(A_{C,i})\quad\text{ for $1\leq i\leq k$.}
\end{equation}
Suppose that $x$ is a nonnegative eigenvector of $A_{D,i}$ with eigenvalue $\rho(A_{D,i})$ for all $1\leq i\leq k$. Take $H:=\{v\in \Lambda^0: v\Lambda D=\emptyset\}$ and $F:=\Lambda^0\backslash(D\cup H)$. Then with respect to the decomposition $\Lambda^0=F\cup D \cup H$, each vertex matrix has block form
\begin{equation}\label{jblock}
A_i=\begin{pmatrix} E_i & B_i&\star\\ 0 & A_{D,i}&0\\0&0&A_{H,i}\end{pmatrix}. 
\end{equation}
where $E_i=A_{F,i}$.
\begin{enumerate}
\item\label{equaly} For $1\leq i,j\leq k$ we have
\begin{equation}\label{defy}
(\rho(A_{D,i})1_F-E_i)^{-1}B_i x=(\rho(A_{D,j})1_F-E_j)^{-1}B_jx.
\end{equation}
\item\label{defy2} Write $y$ for the vector \eqref{defy}. Then $y$ is nonnegative, and for each $i$, $(y,x,0)$ is an eigenvector of $A_i$ with eigenvalue $\rho(A_{D,i})$.
\end{enumerate}
\end{prop}

\begin{proof}
Since $D$ is strongly connected, we have $D\cap H=\emptyset$; since $D$ is hereditary, we must have $D\Lambda H=\emptyset$. Thus the matrices $A_i$ have block form  \eqref{jblock}.

Now fix $i,j\in \{1,\dots ,k\}$. Recall that the factorisation property implies that $A_iA_j=A_jA_i$, and then the block form of the product gives 
$E_iE_j=E_jE_i$ and 
\begin{equation}\label{eq:case2_matrix}
E_iB_j+B_iA_{D,j}=E_jB_i+B_jA_{D,i}. 
\end{equation}
Since $x$ is an eigenvector for both $A_{D,i}$ and $A_{D,j}$, \eqref{eq:case2_matrix} gives
\begin{align}\label{calcivj}
(\rho(A_{D,i})1_F-E_i)B_j x &=\rho(A_{D,i})B_jx-E_iB_jx\\
&=\rho(A_{D,i})B_jx-(E_jB_ix+B_jA_{D,i}x-B_iA_{D,j}x)\notag\\
&=\rho(A_{D,i})B_jx-(E_jB_ix+B_j\rho(A_{D,i})x-B_iA_{D,j}x)\notag\\
&=B_iA_{D,j}x-E_jB_ix\notag\\
&=(\rho(A_{D,j})1_F-E_j)B_ix.\notag
\end{align}
The hypothesis \eqref{Dhered} implies that $\rho(A_{D,i})>\rho(E_i)$, and similarly for $j$. Thus the $F\times F$ matrices $\rho(A_{D,i})1_F-E_i$ and $\rho(A_{D,j})1_F-E_j$ are invertible. Since the matrices $E_i$ and $E_j$ commute, so do $(\rho(A_{D,i})1_F-E_i)^{-1}$ and $(\rho(A_{D,j})1_F-E_j)^{-1}$. So we deduce from \eqref{calcivj} that
\[
(\rho(A_{D,j})1_F-E_j)^{-1}B_jx=(\rho(A_{D,i})1_F-E_i)^{-1}B_ix,
\]
which is part~\eqref{equaly}.

Now write $y$ for the common vector $(\rho(A_{D,i})1_F-E_i)^{-1}B_ix$. Since $x\geq 0$ and $B_i$ has nonnegative entries, the expansion
\begin{align*}
(\rho(A_{D,i})1_F-E_i)^{-1}&=\rho(A_{D,i})^{-1}(1_F-\rho(A_{D,i})^{-1}E_i)^{-1}\\
&=\rho(A_{D,i})^{-1}\sum_{n=0}^\infty\rho(A_{D,i})^{-n}E^n_i
\end{align*}
shows that $y$ is nonnegative. 

We claim that $z:=(y,x,0)$ satisfies $A_iz=\rho(A_{D,i})z$ for $1\le i\le k$.  To see this, we fix $i$ and compute
\[
A_iz=\begin{pmatrix} E_i & B_i&\star\\ 0 & A_{D,i}&0\\0&0&A_{H,i}\end{pmatrix}\begin{pmatrix} y\\ x\\0\end{pmatrix}=\begin{pmatrix} E_i y+B_i x \\ A_{D,i}x\\0\end{pmatrix}.
\]
Since $x$ is an eigenvector for $A_{D,i}$ with eigenvalue $\rho(A_{D,i})$, it suffices to show that $E_i y+B_i x=\rho(A_{D,i}) y$. For this, we compute:
\begin{align*}
E_i y+B_ix 
&=E_i (\rho(A_{D,i})1_F-E_i)^{-1}B_ix +B_ix\\
&=(\rho(A_{D,i})1_F-E_i)^{-1} E_iB_ix+B_i x \\
&=(\rho(A_{D,i})1_F-E_i)^{-1}\big(E_iB_ix +(\rho(A_{D,i})1_F-E_i)B_ix \big)\\
&=(\rho(A_{D,i})1_F-E_i)^{-1}(\rho(A_{D,i})B_i x)\\
&=\rho(A_{D,i})\big( (\rho(A_{D,i})1_F-E_i)^{-1}B_ix\big)\\
&=\rho(A_{D,i})y.\qedhere
\end{align*}
\end{proof}

\begin{rmk}\label{Bxsimul0}
If the graphs $\{\Lambda_C:C\in \Cc\}$ are coordinatewise irreducible, then the relation \eqref{Dhered} imposes some restrictions on the components $B_i$ in the block decomposition \eqref{jblock}. Since every component $C\subset F$ has $C\Lambda D\not=\emptyset$, at least one $B_j$ is nonzero. We claim that $B_i$ is then nonzero for all $i$.

To see this, suppose that $B_i=0$ for some $i$. Then \eqref{eq:case2_matrix} says that $E_iB_j=B_jA_{D,i}$. Lemma~2.1 of \cite{aHLRS2} implies that the matrices $\{A_{D,i}:1\leq i\leq k\}$ have a common Perron-Frobenius eigenvector $x$, and hence
\[
E_i(B_jx)=B_jA_{D,i}x=\rho(A_{D,i})B_jx.
\]
Since $x$ is a Perron-Frobenius eigenvector, it has strictly positive entries, and hence $B_jx$ is not the zero vector. Thus there is a component $C\subset F$ such that $(B_jx)|_C$ is a nonzero eigenvector for $A_{C,i}$ with eigenvalue $\rho(A_{D,i})$. But then $\rho(A_{C,i})\geq \rho(A_{D,i})$, which contradicts \eqref{Dhered}. Thus $B_i$ cannot be~$0$.

As a simple example to illustrate further, consider a $2$-graph with skeleton
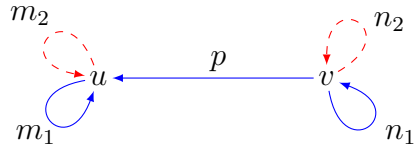
\[
\begin{tikzpicture}[scale=1.5]
 \node[inner sep=0.5pt, circle] (u) at (0,0) {$u$};
    \node[inner sep=0.5pt, circle] (v) at (2,0) {$v$};
\draw[-latex, blue] (u) edge [out=190, in=250, loop, min distance=20, looseness=2.5] (u);
\draw[-latex, red, dashed] (u) edge [out=110, in=170, loop, min distance=20, looseness=2.5] (u);
\draw[-latex, blue] (v) edge [out=280, in=340, loop, min distance=20, looseness=2.5] (v);
\draw[-latex, red, dashed] (v) edge [out=30, in=90, loop, min distance=20, looseness=2.5] (v);
\draw[-latex, blue] (v) edge [out=180, in=0]   (u);
\node at (-0.55, -0.5) {$m_1$};
\node at (-0.6,0.55) {$m_2$};
\node at (2.65, -0.5) {\color{black} $n_1$};
\node at (2.55, 0.5 ) {\color{black} $n_2$};
\node at (1.05, 0.15) {$p$}; 
\end{tikzpicture}
\]
Since the blue-red paths from $v$ to $u$ are in one to one correspondence with the red-blue paths, we have $pn_2=m_2p$ and $m_2=n_2$. In particular, $B_1=(p)\neq 0$ and $B_2=0$.
The only hereditary component is $\{v\}$, but $\rho(A_{\{u\},2})=n_2=\rho(A_{\{v\},2})$, and the graph does not satisfy the hypothesis \eqref{Dhered} for $D=\{v\}$.
\end{rmk}

We now suppose that the matrices $A_{D,i}$ are irreducible. Then the Perron-Frobenius theorem implies that each $A_{D,i}$ has a strictly positive eigenvector with eigenvalue $\rho(A_{D,i})$, and there is a unique  unimodular eigenvector with $\ell^1$-norm 1. Since the matrices $A_{D,i}$ commute,  they share the same unimodular Perron-Frobenius eigenvector \cite[Lemma~2.1]{aHLRS2}.

\begin{prop}\label{existKMS1-modified}
Suppose that $\Lambda$ is a finite $k$-graph without sources or sinks. Consider the preferred dynamics $\alpha^r$ on $\TC^*(\Lambda)$. Suppose that $D$ is a hereditary and nontrivial strongly connected component of $\Lambda^0$ such that $\Lambda_{D}$ is coordinatewise irreducible, and such that $\rho(A_{D,i})>\rho(A_{C,i})$ for all $i$ and components $C\in \Cc\backslash\{D\}$ with $C\Lambda D\not=\emptyset$. Suppose that $x$ is the common unimodular Perron-Frobenius eigenvector of the $A_{D,i}$, take $z=(y,x,0)$ as in Proposition~\ref{evectors}, and write $b:=\|z\|_1>0$. Then there is a KMS$_1$ state $\psi$ of $(\TC^*(\Lambda),\alpha^r)$ such that
\begin{equation}\label{formpsi}
\psi(t_\mu t_\nu^*)=\delta_{\mu,\nu}\rho(A)^{-d(\mu)} b^{-1}z_{s(\mu)}\quad\text{for $\mu,\nu\in \Lambda$,}
\end{equation}
where $\rho(A)^{-d(\mu)}:=\prod_{i=1}^k\rho(A_i)^{-d(\mu)_i}$. 
\end{prop}

\begin{proof}
Choose a decreasing sequence $\{\beta_p\}$ such that $\beta_p\to 1$. Then  we have
\begin{equation}\label{betap>1}
\beta_p r_i=\beta_p\ln\rho(A_i)>\ln\rho(A_i)\quad\text{for all $p$ and $i$.}
\end{equation}
For fixed $p$, we define $y^{\beta_p}_v:=\sum_{\mu\in \Lambda v}e^{-\beta_p r\cdot d(\mu)}$, as in \cite[Theorem~6.1(a)]{aHLRS2}. By Proposition~\ref{evectors}\eqref{defy2}, we have $A_iz=\rho(A_{D,i})z$ for each $i$, and thus 
\[
\prod_{i=1}^k\big(1-e^{-\beta_p r_i}A_i\big)b^{-1}z =\prod_{i=1}^k\big(1-e^{-\beta_p r_i}\rho(A_{D,i})\big)b^{-1}z.
\]
Since $r_i=\ln\rho(A_i)$, and
\[
\rho(A_i)=\max\{\rho(A_{C,i}):C\in \Cc\}, 
\]
we have $\rho(A_{D,i})\leq \rho(A_i)$, and
\begin{align*}
1-e^{-\beta_p r_i}\rho(A_{D,i})&=1-\rho(A_i)^{-\beta_p}\rho(A_{D,i})\\
&\geq 1-\rho(A_i)^{-\beta_p}\rho(A_{i})\\
&=1-\rho(A_i)^{1-\beta_p}.
\end{align*}
It follows from our standing assumptions  in  \S\ref{subsecpreferred} that $\rho(A_i)>1$. Since also  $\beta_p>1$, we have $1-\rho(A_i)^{1-\beta_p}>0$ for all $i$. Thus
\[
\epsilon^p:=\prod_{i=1}^k(1-e^{-\beta_p r_i}A_i)b^{-1}z=\prod_{i=1}^k(1-\rho(A_{D,i}) e^{-\beta_p r_i})b^{-1}z
\]
belongs to $[0,\infty)^{\Lambda^0}$. Since \eqref{betap>1} implies that each $1-e^{-\beta_p r_i}A_i$ is invertible, we can recover
\[
b^{-1}z=\prod_{i=1}^k(1-e^{-\beta_p r_i}A_i)^{-1}\epsilon^p,
\]
and then $\|b^{-1}z\|_1=1$ implies that $\epsilon^p\cdot y^{\beta_p}=1$ (see \cite[Theorem~6.1(a)]{aHLRS2}). Now \cite[Theorem~6.1(b)]{aHLRS2} gives KMS$_{\beta_p}$ states $\psi_{p}$ of $(\TC^*(\Lambda),\alpha^r)$ satisfying
\begin{equation}\label{eq:psi_np}
\psi_{p}(t_\mu t^*_\nu)=\delta_{\mu,\nu} e^{-\beta_p r\cdot d(\mu)}b^{-1}z_{s(\mu)}\quad\text{for $\mu,\nu\in\Lambda$.}
\end{equation}

Since the state space of $\TC^*(\Lambda)$ is weak* compact, we may assume by passing to a subsequence that the sequence $\{\psi_{p}:p\in \N\}$ converges weak* to a state $\psi$ of $\TC^*(\Lambda)$. Letting $p\to \infty$ in \eqref{eq:psi_np} shows that 
\begin{equation*}
\psi(t_\mu t^*_\nu)=\delta_{\mu,\nu}\rho(A)^{-d(\mu)}b^{-1}z_{s(\mu)}\quad\text{for $\mu,\nu\in\Lambda$.}
\end{equation*}
Thus Proposition~3.1(b) of \cite{aHLRS2} implies that $\psi$ is a KMS$_1$ state of $(\TC^*(\Lambda),\alpha^r)$.
\end{proof}

\begin{rmk}
When $\rho(A_i)=\rho(A_{D,i})$ for some $i$, the vector $m^{\psi}=b^{-1}z$ is an eigenvector of $A_i$ with eigenvalue $\rho(A_i)$, and we have $A_im^{\psi}=\rho(A_i)m^{\psi}$. If $\rho(A_i)=\rho(A_{D,i})$ for all $i$, then \cite[Proposition~4.1]{aHLRS2} implies that the KMS$_1$ state $\psi$ factors through a state of $C^*(\Lambda)$.

However, it is possible that $\max\{\rho(A_{C,i}):C\in \Cc\}$ is attained on some component $C$ which lies in the set $H$ of Proposition~\ref{evectors}, and that we then have $\rho(A_{D,i})<\rho(A_i)$. Then we cannot deduce that $\psi$ factors through $C^*(\Lambda)$, even if there exists $j$ such that $\rho(A_{D,j})=\rho(A_j)$. Indeed, since $z|_{D}$ is a Perron-Frobenius eigenvector for $A_{D,i}$, applying Theorem 5.1(c) of \cite{aHKR} to the subalgebra $\TC^*(\Lambda_{D})\subset \TC^*(\Lambda)$ shows that the state $\psi$ does not vanish on the gap projections
\[
\Big\{q_v-\sum_{e\in v\Lambda^{e_i}} t_et_e^*:v\in D\text{ and }\rho(A_{D,i})<\rho(A_i)
\Big\}.
\]
\end{rmk}

We say that a strongly connected component $C$ is \emph{forwards hereditary} if $v\Lambda C\not=\emptyset$ implies $v\in C$. The hypothesis that ``$D$ is not forwards hereditary'' in the next result removes the uninteresting case in which $D$ is disjoint from the rest of the graph, in which case we can study $\Lambda_{D}$ and $\Lambda_{\Lambda\backslash D}$ separately.

\begin{thm}\label{thm:case2-modified}
Suppose that $\Lambda$ is a finite $k$-graph without sources or sinks such that the numbers $\{\ln\rho(A_i):1\leq i\leq k\}$ are rationally independent. Suppose that $D$ is a nontrivial strongly connected component which is hereditary but not forwards hereditary. We suppose further that $\Lambda_{D}$ is coordinatewise irreducible,  and that 
\begin{equation}\label{stronghyp}
\rho(A_{D,i})=\rho(A_i)>\rho(A_{C,i})\quad\text{for all $1\leq i\leq k$ and $C\in \Cc\backslash\{D\}$.}
\end{equation} 
Let $\alpha^r$ be the preferred dynamics on $\TC^*(\Lambda)$, and let $q:\TC^*(\Lambda)\to \TC^*(\Lambda \backslash D)$ be the quotient map. Then  every KMS$_1$ state of $(\TC^*(\Lambda),\alpha^r)$ is a convex combination of the state $\psi$ of Proposition~\ref{existKMS1-modified} and a state $\phi\circ q_D$ lifted from a KMS$_1$ state $\phi$ of $(\TC^*(\Lambda \backslash D),\alpha^r)$.
\end{thm}

We observe that the hypothesis \eqref{stronghyp} is stronger than the corresponding hypothesis \eqref{Dhered} in Propositions~\ref{evectors} and \ref{existKMS1-modified}. This extra strength will be important at the end of the proof when we apply the results of \cite[\S6]{aHLRS2} to the graph $\Lambda\backslash D$.

\begin{proof}
Let $\theta$ be a KMS$_1$ state of $(\TC^*(\Lambda),\alpha^r)$ and consider $m^\theta=(\theta(q_v))$.
We  take $x$ to be the common unimodular Perron-Frobenius eigenvector of the matrices $A_{D,i}$, and apply Proposition~\ref{evectors}. We write the vertex matrices in block form with respect to the decomposition $\Lambda^0=(\Lambda^0\backslash D)\cup D$ as
\begin{equation}\label{2block}
A_i=\begin{pmatrix} E_i & B_i\\ 0 & A_{D,i}\end{pmatrix}; 
\end{equation}
notice that here we have absorbed the set $H$ of Proposition~\ref{evectors} into $\Lambda^0\backslash D$. So now we extend the vector $y\in [0,\infty)^F$ of Proposition~\ref{evectors}\eqref{defy2} to a vector $y\in [0,\infty)^{\Lambda^0\backslash D}$ by setting $y_v=0$ if $v\in H$, or equivalently if $v\Lambda D=\emptyset$. Then we set $z=(y,x)$ and $b=\|z\|_1$, and let $\psi$  be the state of Proposition~\ref{existKMS1-modified}. 

Since the dynamics is preferred, Proposition~4.1 of \cite{aHLRS2} with $K=\{i\}$ gives 
\[
A_im^\theta\leq \rho(A_i)m^\theta\quad\text{for $1\leq i\leq k$,}
\]
and the block decomposition implies that
\[
A_{D,i}\big(m^\theta|_{D}\big)\leq \rho(A_i)m^\theta|_{D}\quad\text{for $1\leq i\leq k$.}
\]
Since $\rho(A_{i})=\rho(A_{D,i})$ for all $i$, the subinvariance theorem \cite[Theorem~1.6]{Sen} implies that
\[
A_{D,i}\big(m^\theta|_{D}\big)=\rho(A_{D,i})m^\theta|_{D}=\rho(A_{i})m^\theta|_{D}\quad\text{for $1\leq i\leq k$.}
\]
(The subinvariance theorem requires that $m^\theta|_{D}\not=0$, but our conclusion also holds trivially if $m^\theta|_{D}=0$.) Thus $m^\theta|_{D}$ is a multiple of the common unimodular Perron-Frobenius eigenvector $x$ for the family $\{A_{D,i}:1\leq i\leq k\}$, and we can define  $a\in [0,\infty)$ by  $m^{\theta}|_{D}=ab^{-1}x$. Let $\psi$ be the state of Proposition~\ref{existKMS1-modified}. Then, in particular, for $v\in D$ we have $a\psi(q_v)=\theta(q_v)$. 

If $a=0$, we have $m^\theta_v=\theta(q_v)=0$ for all $v\in D$, and an application of \cite[Lemma~2.2]{aHLRS1} shows that $\theta$ factors through a state of $\TC^*(\Lambda\backslash D)$. So we suppose that $a>0$. Next we want to show that $a\leq 1$, which we do by proving that $\theta(q_v)\geq a\psi(q_v)$ for all $v\in \Lambda^0$. Since  $\theta(q_v)=a\psi(q_v)$ for $v\in D$,  we consider $v\in \Lambda^0\backslash D$. If $v\Lambda D=\emptyset$, then $y_v=0$, and hence $\theta(q_v)\geq 0=a\psi(q_v)$. So we suppose that $v\Lambda D\not=\emptyset$.

We fix $j$,  and work in the coordinate graph $\Lambda_j$ using techniques from the proof of \cite[Theorem~4.3(b)]{aHLRS4}. As there, we consider the set
\[
\QE_j:=\QE_j(D):=\{\mu e\in \Lambda^{\N e_j}\Lambda^{e_j}D:r(e)\notin D\}
\]
of $j$-coloured paths which make a quick exit from $D$. By Lemma~4.4 of \cite{aHLRS4}, the projections $\{t_\lambda t_\lambda^*:\lambda\in v\QE_j\}$ are mutually orthogonal, and hence
\begin{align*}
\theta(q_v)&\geq \sum_{\lambda\in v\QE_j}\theta(t_\lambda t^*_\lambda)\\
&=\sum_{\lambda\in v \QE_j}\rho(A_j)^{-|\lambda|}\theta(q_{s(\lambda)})\\
&=\sum_{w\in D}\sum_{\lambda\in v \QE_jw}\rho(A_j)^{-\vert \lambda \vert} ab^{-1}x_{w}\\
&=\sum_{w\in D}\sum_{m=0}^\infty\sum_{u\in\Lambda^0\backslash D} \rho(A_j)^{-(m+1)}A_j^m(v,u)A_j(u,w) ab^{-1}x_{w}.
\end{align*}
Now we recall the block decomposition of $A_j$, which says that for $u\in \Lambda^0\backslash D$ we have $A^m_j(v,u)=E^m_j(v,u)$, and for $w\in D$ we have $A_j(u,w)=B_j(u,w)$. Thus
\begin{align*}
\theta(q_v)& \ge\sum_{w\in D}
 \rho(A_j)^{-1}\Big(\sum_{u\in\Lambda^0\backslash D}\sum_{m=0}^\infty \rho(A_j)^{-m}E_j^m(v,u)B_j(u,w)ab^{-1}x_w\Big)\\
&=\sum_{w\in D}
 \rho(A_j)^{-1}\Big(\sum_{u\in\Lambda^0\backslash D}\big(1-\rho(A_j)^{-1}E_j\big)^{-1}(v,u)B_j(u,w)ab^{-1}x_w\Big)\notag\\
&=\sum_{w\in D}\big((\rho(A_j)-E_j)^{-1}B_j\big)(v,w)ab^{-1}ßx_w.\notag
\end{align*}
Since $\rho(A_j)=\rho(A_{D,j})$, we can recognise this sum as $ab^{-1}$ times that defining the coordinate $y_v$ of the vector $y$ of Proposition~\ref{evectors}, and hence
\begin{equation*}
\theta(q_v) \ge ay_v =a\psi(q_v)\quad\text{for all $v\in \Lambda^0\backslash D$ such that $v\Lambda D\neq\emptyset$.}
\end{equation*}
We have now shown that $\theta(q_v)\geq a\psi(q_v)$ for all $v\in \Lambda^0$.
We deduce   that $1=\theta(1)\geq a\psi(1)=a$, which is what we wanted to show.

If $a=1$, then the termwise inequality 
\[
1=\sum_{v\in \Lambda^0}\theta(q_v)\geq \sum_{v\in \Lambda^0} a\psi(q_v)=\sum_{v\in \Lambda^0} \psi(q_v)
\]
forces $\theta(q_v)=\psi(q_v)$ for all $v$. Since both are KMS$_1$ states and $r$ has rationally independent coordinates,  
Proposition 3.1(b) of \cite{aHLRS2} implies that $\theta(t_\lambda t_\mu^*)=\psi(t_\lambda t_\mu^*)$ for all $\lambda,\mu\in\Lambda$, and $\theta=\psi$. 

The other possibility is that $0<a<1$.  Then  $\theta(q_v)-a\psi(q_v)\geq 0$ for all $v\in \Lambda^0$. Since $\theta(q_v)=a\psi(q_v)$ for $v\in D$, 
\begin{equation}\label{eq-def-kappa}
\kappa_v:= {(1-a)}^{-1}(m^\theta-am^\psi)_v={(1-a)}^{-1}\big(\theta(q_v)-a\psi(q_v)\big)
\end{equation}
defines a vector $\kappa$ in $[0,\infty)^{\Lambda^0\backslash D}$ with $\|\kappa\|_1=1$. Using the notation of the block decomposition \eqref{2block}, we claim that the vector
\[
\epsilon:=\prod_{i=1}^k\big(1-\rho(A_i)^{-1}E_i\big)\kappa
\]
belongs to $[0,\infty)^{\Lambda^0\backslash D}$.

Since $\theta$ is a KMS$_1$ state, Proposition~4.1 of \cite{aHLRS2} implies that
\begin{equation}\label{subinvi}
\prod_{i=1}^k\big(1-\rho(A_i)^{-1}A_i\big)m^\theta\geq 0.
\end{equation}
Writing $A_i$ in block form \eqref{2block}, we can rewrite \eqref{subinvi} as
\[
\prod_{i=1}^k\begin{pmatrix}
1-\rho(A_i)^{-1}E_i&-\rho(A_i)^{-1}B_i\\
0&1-\rho(A_i)^{-1}A_{D,i}\end{pmatrix}\begin{pmatrix}m^\theta|_{\Lambda\backslash D}\\
m^\theta|_D\end{pmatrix}\geq 0.
\]
Since $m^\theta|_D=ab^{-1}x$ is an eigenvector for $A_{D,k}$ with eigenvalue $\rho(A_{D,k})=\rho(A_k)$, the bottom block is $0$. Expanding the top block gives
\begin{equation}\label{upperblock}
\prod_{i=1}^k\big(1-\rho(A_i)^{-1}E_i\big)\big(m^\theta|_{\Lambda\backslash D}\big)-\prod_{i=1}^{k-1}\big(1-\rho(A_i)^{-1}E_i\big)\rho(A_k)^{-1}B_k(ab^{-1}x)\geq 0.
\end{equation}
Since the hypothesis \eqref{stronghyp} implies that 
\[
\rho(A_i)>\rho(E_i)=\max\{\rho(A_{C,i}):C\in \Cc,\ C\not= D\},
\] 
the matrices $1-\rho(A_i)^{-1}E_i$ are all invertible. Thus we can rearrange the second term in \eqref{upperblock} as
\begin{align*}
\prod_{i=1}^{k-1}\big(1-\rho(A_i)^{-1}&E_i\big)\rho(A_k)^{-1}B_k(ab^{-1}x)\\
&=\prod_{i=1}^k\big(1-\rho(A_i)^{-1}E_i\big)\big(1-\rho(A_k)^{-1}E_k\big)^{-1}\rho(A_k)^{-1}B_k(ab^{-1}x)\\
&=\prod_{i=1}^k\big(1-\rho(A_i)^{-1}E_i\big)ab^{-1}\big(\rho(A_k)1-E_k)^{-1}B_kx.
\end{align*}
Since $\rho(A_k)=\rho(A_{D,k})$, we can recognise $\big(\rho(A_k)1-E_k)^{-1}B_kx$ as the vector $y$. Hence
\begin{align*}
\prod_{i=1}^{k-1}\big(1-\rho(A_i)^{-1}E_i\big)\rho(A_k)^{-1}B_k(ab^{-1}x)
&=\prod_{i=1}^k\big(1-\rho(A_i)^{-1}E_i\big)ab^{-1}y\\
&=\prod_{i=1}^k\big(1-\rho(A_i)^{-1}E_i\big)\big(am^\psi|_{\Lambda\backslash D}\big).
\end{align*}
Combining this with \eqref{upperblock} gives 
\[
\prod_{i=1}^k\big(1-\rho(A_i)^{-1}E_i\big)\big(m^\theta|_{\Lambda\backslash D}-am^\psi|_{\Lambda\backslash D}\big)\geq 0,
\]
and now
\begin{align*}
\epsilon
&=(1-a)^{-1}\prod_{i=1}^k\big(1-\rho(A_i)^{-1}E_i\big)\big(m^\theta|_{\Lambda\backslash D}-am^\psi|_{\Lambda\backslash D}\big)\geq 0,
\end{align*}
and hence $\epsilon$ belongs to $[0,\infty)^{\Lambda^0\backslash D}$,  as claimed.

Since  each $1-\rho(A_i)^{-1}E_i$ is invertible, we  can recover  $\kappa=\prod_{i=1}^k\big(1-\rho(A_i)^{-1}E_i\big)^{-1}\epsilon$, and since $\|\kappa\|_1=1$, it follows from \cite[Theorem~6.1(a)]{aHLRS2} that $\epsilon$ belongs to the simplex $\Sigma_1$ of \cite[Theorem~6.1(c)]{aHLRS2} for the graph $\Lambda\backslash D$. The induced dynamics $\alpha^r$ on $\TC^*(\Lambda\backslash D)$ satisfies $r_i=\ln\rho(A_i)>\ln\rho(E_i)$. Thus  Theorem~6.1 of \cite{aHLRS2} gives a KMS$_1$ state $\phi_\epsilon$ on $\TC^*(\Lambda \backslash D)$ such that $\phi_\epsilon( q_v)=\kappa_v$ for $v\in\Lambda^0\backslash D$. (The graph $\Lambda\backslash D$ could have sources, so we are implicitly using here that we can apply \cite[Theorem~6.1]{aHLRS2} to graphs with sources, as discussed in \S\ref{sources}.) 

Let $q:\TC^*(\Lambda)\to \TC^*(\Lambda\backslash D)$ be the quotient map. Looking back at the definition of $\kappa$ at \eqref{eq-def-kappa}, we see that 
\begin{align*}
\big((1-a)\phi_\epsilon\circ q+a\psi\big)(q_v)&=(1-a)\kappa_v+a\psi(q_v)=\theta(q_v),
\end{align*}
for $v\in \Lambda^0\backslash D$. For $v\in D$, we have 
\[
\big((1-a)\phi_\epsilon\circ q+a\psi\big)(q_v)=a\psi(q_v)=\theta(q_v).
\]
Thus the convex combination $(1-a)\phi_\epsilon\circ q+a\psi$ agrees with $\theta$ on all the vertex projections, and since they are both KMS$_1$ states, another application of \cite[Proposition~3.1]{aHLRS2} shows that they are equal on $\TC^*(\Lambda)$. Thus $\theta$ is a convex combination of $\phi_\epsilon\circ q$ and the state $\psi$ of Proposition~\ref{existKMS1-modified}. This completes the proof of Theorem~\ref{thm:case2-modified}.
\end{proof}

\begin{rmk}
We observe that the full strength of the subinvariance relation~\eqref{subinvi} is needed to prove that $\epsilon\geq 0$. In \cite[Proposition~4.1]{aHLRS2} we proved a similar subinvariance relation for all nonempty subsets $K$ of $\{1, \dots, k\}$.
\end{rmk}
\begin{rmk}
The requirement that $\rho(A_{D,i})=\rho(A_i)>\rho(A_{C,i})$ for all $i$  and $C\neq D$ is crucial at the end of the above proof, because it allows us to recover the vector $\kappa$ from $\epsilon$. We have so far been unable to find an approximation argument like that of Proposition~\ref{existKMS1-modified} which will give us a suitable KMS$_1$ state as a limit of KMS$_\beta$ states with $\beta>1$.
\end{rmk}

Since the results of \cite{aHLRS2} give a complete classification of the KMS$_1$ states of the quotient $\TC^*(\Lambda \backslash D)$ (see \S\ref{sources}), we have the following description of the simplex of KMS$_1$ states on $\TC^*(\Lambda)$. Recall that for large enough $\beta$, Theorem~6.1 of \cite{aHLRS2} describes the simplex of KMS$_\beta$ states in terms of a vector $y^\beta:\Lambda^0\backslash D\to [1,\infty)$ as
\[
\Sigma_\beta:=\{\epsilon\in [0,\infty)^{\Lambda^0\backslash D}:\epsilon\cdot y^\beta=1\}
\]
and gives a specific formula for the KMS$_\beta$ state $\phi_{\epsilon}$ associated to $\epsilon\in \Sigma_\beta$ (see formula (6.1) in \cite{aHLRS2}). The critical value $\beta=1$ is large enough for the preferred dynamics on $\TC^*(\Lambda \backslash D)$. So we have:

\begin{cor}\label{cor-main}
Suppose that $\Lambda$ is a finite $k$-graph without sources or sinks, and $D$ is a nontrivial strongly connected component which is hereditary, for which $\Lambda_D$ is cocordinatewise irreducible, and which has $\rho(A_{D,i})>\rho(A_{C,i})$ for $1\leq i\leq k$ and all other components $C\in \Cc\backslash\{D\}$. Let $q$ be the quotient map of $\TC^*(\Lambda)$ onto $\TC^*(\Lambda\backslash D)$, and let $\psi$ be the state of Proposition~\ref{existKMS1-modified}. Then the map $(s\epsilon, 1-s)\mapsto s(\phi_\epsilon\circ q)+(1-s)\psi$ is an affine homeomorphism of the simplex
\[
\Sigma_1':=\{(s\epsilon,1-s):\epsilon\cdot y^1=1\text{ and }s\in [0,1]\}
\]
onto the simplex of KMS$_1$ states on $(\TC^*(\Lambda),\alpha^r)$.
\end{cor}

\begin{rmk}
We prove later in  Proposition~\ref{KMSsources} below that the state $\phi_\epsilon$ of $\TC^*(\Lambda \backslash D)$ factors through a state of $C^*(\Lambda \backslash D)$ if and only if $\epsilon$ is supported on the set of ``absolute sources'' in $\Lambda \backslash D$. (See Remark~\ref{rmk-absolute-sources} for the definition of ``absolute source''.) The graph $\Lambda \backslash D$ could certainly have absolute sources, as taking $D=\{x\}$ in Example~\ref{creatingsources} shows. In general,  if $\epsilon$ is supported on an absolute source $v\in \Lambda^0 \backslash D$, then the state $\phi_\epsilon$ of $\TC^*(\Lambda)$ satisfies $\phi_\epsilon(q_v)>0$. Since the sources in $C^*(\Lambda \backslash D)$ belong to the saturation $\Sigma D$ of the hereditary set $D\subset \Lambda^0$, the associated vertex projections $q_v\in \TC^*(\Lambda)$ are killed by the quotient map of $\TC^*(\Lambda)$ onto $C^*(\Lambda\backslash\Sigma D)$. So the only KMS$_1$ state of $\TC^*(\Lambda)$ that factors through a state of $C^*(\Lambda\backslash\Sigma D)$ is the state $\psi$ in the statement of Corollary~\ref{cor-main}.
\end{rmk}

\section{Graphs with one or two components}

Suppose that $\Lambda$ is a finite $k$-graph with no sinks or sources, and as usual write $\Cc$ for the set of nontrivial strongly connected components. In this section we assume further that
\begin{itemize}
\item[(A1)] $\Lambda$ has at most two nontrivial strongly connected components;
\item[(A2)] for every $C\in \Cc$, the graph $\Lambda_C:=C\Lambda C$ is  coordinatewise irreducible.
\end{itemize}
Because the graph is finite and has no sinks, each $\Lambda^nv$ is nonempty, and there exist paths which visit some vertices more than once. So there is always at least one nontrivial strongly connected component.

The point of assumption (A2) is that the strongly connected components are the same as the strongly connected components in each of the coordinate graphs $\Lambda_i=(\Lambda^0,\Lambda^{e_i},r,s)$. In general, this is not necessarily the case:

\begin{example}\label{examplenotcoordirred}
There are $2$-graphs with the following skeleton (drawn with the convention of~\S\ref{hrgraphs}). Since we made the numbers of red and blue loops at $u$ and $v$ the same, there are $2$-graphs with this skeleton, but the results of this section do not apply to such graphs. The set $C=\{u,v\}$ is strongly connected, but the graph does not satisfy (A2) because $A_2$ is reducible.
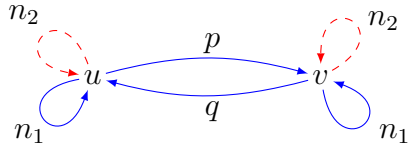
\[
\begin{tikzpicture}[scale=1.5]
 \node[inner sep=0.5pt, circle] (u) at (0,0) {$u$};
    \node[inner sep=0.5pt, circle] (v) at (2,0) {$v$};
\draw[-latex, blue] (u) edge [out=190, in=250, loop, min distance=20, looseness=2.5] (u);
\draw[-latex, red, dashed] (u) edge [out=110, in=170, loop, min distance=20, looseness=2.5] (u);
\draw[-latex, blue] (v) edge [out=280, in=340, loop, min distance=20, looseness=2.5] (v);
\draw[-latex, red, dashed] (v) edge [out=30, in=90, loop, min distance=20, looseness=2.5] (v);
\draw[-latex, blue] (v) edge [out=195, in=345]   (u);
\draw[-latex, blue] (u) edge [out=15, in=165]   (v);
\node at (-0.55, -0.5) {$n_1$};
\node at (-0.6,0.55) {$n_2$};
\node at (2.65, -0.5) {\color{black} $n_1$};
\node at (2.55, 0.5 ) {\color{black} $n_2$};
\node at (1.05, 0.3) {$p$}; 
\node at (1.05, -0.3) {$q$};
\end{tikzpicture}
\]
\end{example}

\subsection{Graphs with one component}\label{sec1cpt} We now suppose in addition that there is exactly one nontrivial component $C$. Since there are no sinks or sources, every vertex emits and receives edges of all degrees, and then paths of sufficiently large degrees must hit some vertices more than once. So every vertex connects forwards and backwards to $C$, and thus $C=\Lambda^0$. Assumption (A2) therefore implies that $\Lambda$ is coordinatewise irreducible. Theorem~7.2 of \cite{aHLRS2} says that if the numbers $\{\ln\rho(A_i):1\leq i\leq k\}$ are rationally independent, then the preferred dynamics  on $\TC^*(\Lambda)$ admits a unique KMS$_1$ state, which then factors through $C^*(\Lambda)$. (This result is substantially improved in \cite[\S7]{aHLRS3}, but the original one will suffice here.)

We shall be interested in KMS states on quotients of Toeplitz algebras, and although these quotients are themselves Toeplitz algebras, the dynamics on the quotient induced by the preferred dyanmics may not be itself preferred. For a non-preferred dynamics on $\TC^*(\Lambda)$, we can find KMS$_\beta$ states for all $\beta>\beta_c:=\max_i\{r_i^{-1}\ln\rho(A_i)\}$ using \cite[Theorem~6.1]{aHLRS2}. There is always at least one KMS$_{\beta_c}$ state of $\TC^*(\Lambda)$, and if the numbers $\{\ln\rho(A_i):1\leq i\leq k\}$ are rationally independent, this is the only KMS$_{\beta_c}$ state \cite[Theorem~5.1]{aHKR}. However, it typically does not factor through $C^*(\Lambda)$ (see \cite[Proposition~6.1]{aHKR}). 

The KMS states on the Toeplitz algebras of $2$-graphs with a single vertex were explicitly described in \cite[\S7]{aHKR}. In that case the vertex matrices are $(m_1)$ and $(m_2)$, and the rational independence of $\ln m_1$ and $\ln m_2$ is automatic unless $m_1$ and $m_2$  have the form described in Proposition~\ref{ricond}.

\subsection{Graphs with two components}\label{sec:setup} We now suppose that $\Lambda$ has exactly two nontrivial components $C$ and $D$. We assume further (that is, in addition to (A1) and (A2) above) that
\begin{itemize}
\item[(A3)] $C\Lambda v\not=\emptyset$ for all $v\in \Lambda^0$.
\end{itemize}
This assumption has some immediate consequences. First, $C$ has to be forwards hereditary (that is, $v\Lambda C\not=\emptyset$ implies $v\in C$) and $\Lambda^0\backslash C$ is hereditary. Since $C\Lambda D\not=\emptyset$, the existence of two distinct components implies that $D\Lambda C=\emptyset$. There are some less immediate consequences too.

\begin{lem}\label{lempaths}
Suppose that $1\leq j\leq k$. If $w\in \Lambda^0\backslash(C\cup D)$, then there are paths $\lambda$ in $C\Lambda^{\N e_j}w$ and $\mu$ in $w\Lambda^{\N e_j}D$. 
\end{lem}

\begin{proof}
Take $N>|\Lambda^0|$. Since $\Lambda$ has no sinks, $\Lambda^{Ne_j}w$ is nonempty, and there exists $\lambda\in \Lambda^{Ne_j}w$. Since $\lambda$ passes through $N+1$ vertices, it must pass through at least one several times. Thus $\lambda$ contains a return path. Similarly, since $\Lambda$ has no sources, there exists $\mu\in w\Lambda^{Ne_j}$, and it too contains a return path. Each of the return paths in $\lambda$ and $\mu$ must lie in one of the strongly connected components $\Lambda_{C}$ or $\Lambda_{D}$. They cannot lie in the same one, because then $w$ would be strongly connected to that component, and would belong to it. If the return path in $\lambda$ lies in $\Lambda_{D}$ and the one on $\mu$ lies in $\Lambda_{C}$, then we have a path in $\Lambda_{D}\Lambda w\Lambda\Lambda_{C}$, which contradicts $D\Lambda C=\emptyset$. So the return path in $\lambda$ must lie in $\Lambda_{C}$, and the one in $\mu$ must lie in $\Lambda_{D}$. But then $\lambda$ ends in $C$ and $\mu$ begins in $D$.
\end{proof}

\begin{cor}
 The subset $D$ is hereditary.
\end{cor}

\begin{cor}\label{allcols}
If $C\cup D$ is a proper subset of $\Lambda^0$, then there are paths of all colours from  $D$ to $C$.
\end{cor}

\begin{rmk}
It is important in Corollary~\ref{allcols} that there is a vertex $w\in \Lambda^0\backslash (C\cup D)$ to which we can apply Lemma~\ref{lempaths}. If $C\cup D=\Lambda^0$, then it is possible that there are only edges of one colour from $D=\{v\}$ to $C=\{u\}$. (The graphs in Example~\ref{examplenotcoordirred} with $p=0$, for example.)
\end{rmk}

We now summarise Proposition~\ref{utstructure} as it applies to our two component graphs.

\begin{lem}\label{lem:matrix_decom}
Suppose we have $\Lambda$ as above and $C\Lambda v\not=\emptyset$ for all $v\in \Lambda^0$. Then we can order the vertex set $\Lambda^0$ so that every vertex matrix  has the block form
\[
A_j=\begin{pmatrix}
A_{C,j}&\star&\star\\
0&B_j&\star\\
0&0&A_{D,j}
\end{pmatrix}
\]
with each $B_j$ strictly upper triangular. In particular, we have
\begin{equation*}
\rho(A_j)=\max\{\rho(A_{C,j}),\rho(A_{D,j})\}.
\end{equation*}
\end{lem}

\subsection{KMS states for graphs with two components}\label{KMS2cpts}

We consider a finite graph $\Lambda$ with no sinks or sources satisfying all the assumptions of \S\ref{sec:setup}, and we use the same notation.  This gives two strongly connected components $C$ and $D$ such that 
\begin{itemize}
\item $C$ is  forwards hereditary and  $D$ is hereditary,
\item $\Lambda_C$ and $\Lambda_D$ are coordinate-wise irreducible, and 
\item the hereditary closure of $C$ is $\Lambda^0$.
\end{itemize}
We take $r_i=\ln\rho(A_i)$, so that $\alpha^r$ is the preferred dynamics on $\TC^*(\Lambda)$.
If $\beta>1$, then $\beta r_i>\ln\rho(A_i)$ for all $i$, and Theorem~6.1 of \cite{aHLRS2} gives  a $(|\Lambda^0|-1)$-dimensional simplex of KMS$_\beta$ states on $(\TC^*(\Lambda),\alpha^r)$. 

First consider $\Lambda$ such that $C$ is $j$-critical for some $j$.  Then Proposition~\ref{boundbeta}  implies that every KMS$_\beta$ state of $(\TC^*(\Lambda),\alpha^r)$ has $\beta\geq 1$.   So in this situation, it remains to consider $\beta=1$. We want to apply Proposition~\ref{prop:Cminimal-modified}, but  the set $\{w\in\Lambda^0: C\Lambda^{\N e_j} w\neq\emptyset\}$ appearing there may not be all of $\Lambda^0$ (see for example, the $2$-graph on page~\pageref{graph-page-ref}). 
But if, for example, $C\cup D$ is a proper subset of $\Lambda^0$, then
 Corollary~\ref{allcols} implies there are paths of all colours from $D$ to $C$, and then $\{w\in\Lambda^0: C\Lambda^{\N e_j} w\neq\emptyset\}=\Lambda^0$. So we also assume that 
 \begin{itemize}
\item $\{w\in\Lambda^0: C\Lambda^{\N e_j} w\neq\emptyset\}=\Lambda^0$.
\end{itemize}
We let  $K_C=\{i:\rho(A_i)=\rho(A_{C,i})\}$ and observe it is nonempty because $C$ is $j$-critical.
Now by Proposition~\ref{prop:Cminimal-modified} every KMS$_1$ state of $(\TC^*(\Lambda),\alpha^r)$ factors through $(\TC^*(\Lambda_C),\alpha^r)$, and also through  the quotient by the ideal $I_{K_C}$ generated by the gap projections
\[
q_v-\sum_{e\in v\Lambda_C^{e_i}} t_e t^*_e\quad\text{ for $i\in K_C$.}
\]
If  $K_C=\{1,\dots,k\}$, this quotient is the graph algebra $C^*(\Lambda_C)$  by Corollary~\ref{corpreferred-modified}.
Proposition~\ref{prop:Cminimal-modified} also says that if the numbers $\{\ln\rho(A_i):1\leq i\leq k\}$ are rationally independent, then there is exactly one such state on $(\TC^*(\Lambda),\alpha^r)$. 

Second, we consider $\Lambda$ where $C$ is not $i$-critical for any $i$, that is, $K_C=\emptyset$. Then Lemma~\ref{lem:matrix_decom} implies that
\[
\rho(A_i)=\max\{\rho(A_{C,i}),\rho(A_{D,i})\}=\rho(A_{D,i})>\rho(A_{C,i})\quad\text{for all $i$.}
\]
By Theorem~\ref{thm:case2-modified}, the KMS$_1$ states of $(\TC^*(\Lambda),\alpha^r)$ are convex combinations of the state $\psi$ of Proposition~\ref{existKMS1-modified} and a state of the form $\phi\circ q_D$ for some KMS$_1$ state $\phi$ of $\TC^*(\Lambda\backslash D)$. Since 
\[
r_i=\ln\rho(A_i)>\ln\rho(A_{C,i})=\ln\rho(A_{\Lambda\backslash D,i})\quad\text{for all $i$},
\]
Theorem~6.1 of \cite{aHLRS2} gives an explicit description $\{\phi_\epsilon:\epsilon\in \Sigma_1\}$ of a $(|\Lambda^0|-|D|-1)$-dimensional simplex of KMS$_1$ states of $\TC^*(\Lambda\backslash D)$. (Notice that, as Example~\ref{creatingsources} illustrates, the graph $\Lambda\backslash D$ may have sources, so we are using here that we can apply \cite[Theorem~6.1]{aHLRS2} to graphs with sources, as discussed in \S\ref{sources}.) Thus the simplex of KMS$_1$ states of $\TC^*(\Lambda)$ has dimension $|\Lambda^0\backslash D|$. By Proposition~\ref{KMSsources} the states of the form $\phi_\epsilon\circ q_D$ factor through states of $C^*(\Lambda\backslash D)$ if and only if  $\epsilon$ is supported on the ``absolute sources''  in $\Lambda\backslash D$, which under our hypotheses all lie outside $C$.

When $K_C$ is empty, the system $(\TC^*(\Lambda),\alpha^r)$ has a further phase transition at
\begin{equation}\label{betac}
\beta_c:=\max_i\{r_i^{-1}\ln\rho(A_{C,i}):1\leq i\leq k\}.
\end{equation}
For $\beta_c<\beta<1$, \cite[Theorem~6.1]{aHLRS2} gives a $(|\Lambda^0|-|D|-1)$-dimensional simplex of KMS$_\beta$ states on $(\TC^*(\Lambda\backslash D),\alpha^r)$, and composing with $q_D$  gives a simplex of KMS$_\beta$ states on $(\TC^*(\Lambda),\alpha^r)$. 

\begin{rmk}
Suppose that $\Lambda^0=C\cup D$ (plus all the assumptions used above). Then $\Lambda_{\Lambda\backslash D}=\Lambda_C$. 
At this point we have a complete analysis of the KMS states of $(\TC^*(\Lambda),\alpha^r)$ for $\beta>\beta_c$ because we know all the KMS states of the coordinatewise irreducible graph $\Lambda_C$ from \cite{aHKR}. 

However, when $\Lambda^0\neq C\cup D$, there are vertices which lie between $D$ and $C$, and then $\Lambda\backslash D$ has sources (see Example~\ref{creatingsources}), and we need to work a little harder to  deal with this.
\end{rmk}

To see what happens at the (second) critical inverse temperature $\beta_c$ of \eqref{betac}, we need a lemma: notice that we want to apply this lemma for a dynamics that is \emph{not} the preferred dynamics.

\begin{lem}\label{lem-2ndcritical}
Suppose that $\Gamma$ is a finite $k$-graph with a single nontrivial strongly connected component $C$, that $\Gamma$ has no sinks,  and that $\Gamma_C$ is coordinatewise irreducible with the numbers $\{\rho(A_{C,i}):1\leq i\leq k\}$ rationally independent. Take $r\in (0,\infty)^k$, and $\beta_c$ as in \eqref{betac}. Then $(\TC^*(\Gamma),\alpha^r)$ has a unique KMS$_{\beta_c}$ state $\phi$. It satisfies $\phi(q_v)=0$ for all $v\in \Gamma^0\backslash C$, and factors through a state of $C^*(\Gamma_C)$.
\end{lem}

\begin{proof}
Suppose that $\phi$ and $\phi'$ are KMS$_{\beta_c}$ states on $(\TC^*(\Gamma),\alpha^r)$. We choose $j$ such that $\beta_c=r_j^{-1}\ln\rho(A_{C,j})$. Let $\Gamma_j$ be the directed graph $(\Gamma^0,\Gamma^{e_j},r,s)$. The family 
\[
\{q_v,t_e:v\in \Gamma^0,\ e\in \Gamma^{e_j}\}\subset \TC^*(\Gamma)
\]
is a Toeplitz-Cuntz-Krieger $\Gamma_j$-family in $\TC^*(\Gamma)$. If $v$ is not a source, then each $q_v-\sum_{r(e)=v}t_et_e^*$ is nonzero (consider the finite path representation of $\TC^*(\Gamma)$), and Theorem~4.1 of \cite{FR} implies that $\pi_{q,t}$ is an injective homomorphism of $\TC^*(\Gamma_j)$ into $\TC^*(\Gamma)$. The isomorphism is equivariant for the dynamics $\alpha^{r_j}$ on $\TC^*(\Gamma_j)$ given by $\alpha^{r_j}_t=\gamma_{e^{ir_jt}}$ and the given dynamics $\alpha^r$ on $\TC^*(\Gamma)$. 

For the usual dynamics $\alpha$ on the graph algebra $\TC^*(\Gamma_j)$, the proof of Corollary~6.1(b) of \cite{aHLRS1}  implies that $(\TC^*(\Gamma_j),\alpha)$ has a unique KMS$_{\ln\rho(A_{C,j})}$ state and that this state vanishes on the vertex projections $q_v$ for $v\in \Gamma^0\backslash C$.  Scaling the dynamics, in this case inserting the scale $r_j$, changes the inverse temperature to $r_j^{-1}\ln\rho(A_{C,j})$, which is $\beta_c$ (see \cite[Lemma~2.1]{aHKR}). Thus $\phi\circ\pi_{q,t}=\phi'\circ\pi_{q,t}$, which in particular implies that $\phi(q_v)=\phi'(q_v)$ for all $v\in \Gamma^0$. Since both $\phi$ and $\phi'$ are KMS states and the numbers $\{\rho(A_{C,i})\}$ are rationally independent, they therefore also agree on all the spanning elements $t_\mu t_\lambda^*$, and hence on $\TC^*(\Gamma)$. (We need rational independence to see that $\phi(t_\mu t_\lambda^*)=0$ when $\mu\not=\lambda$ using \cite[Proposition~3.1(b)]{aHLRS2}.)  
Thus $(\TC^*(\Gamma),\alpha^r)$ has a unique KMS$_{\beta_c}$ state $\phi$, and $\phi(q_v)=0$ for $v\in \Gamma^0\backslash C$. That $\phi$ factors through a state of $C^*(\Gamma_C)$ follows from \cite[Theorem~7.2]{aHLRS2}.
\end{proof}

We now return to the preceding discussion of the case where the set $K_C$ is empty. Then Lemma~\ref{lem-2ndcritical} applies to the graph $\Gamma=\Lambda\backslash D$, and shows that there is a unique KMS$_{\beta_c}$ state on $\TC^*(\Lambda\backslash D)$, and hence also on $\TC^*(\Lambda)$. This state factors through the quotient map onto $C^*(\Lambda_C)$. This completes the description of the KMS states for the preferred dynamics on the Toeplitz algebras of $k$-graphs with two strongly connected components (with quite a few assumptions on the graphs).

\section{Graphs with sources}\label{sources}

The analysis in \cite{aHLRS2} of KMS states on the algebras of higher-rank graphs carries the hypothesis that ``$\Lambda$ is a finite $k$-graph without sources''. However, that analysis does not actually use the ``without sources'' hypothesis in the main results for states on the Toeplitz algebra $\TC^*(\Lambda)$ \cite[Theorem~6.1]{aHLRS2}. The hypothesis does become important when we ask which KMS states of $\TC^*(\Lambda)$ factor through states of $C^*(\Lambda)$. As Kajiwara and Watatani \cite{KW} showed for a $1$-graph $E$, the presence of sources then makes a big difference: in the notation of \cite{aHLRS1}, a KMS state $\phi_\epsilon$ of $\TC^*(E)$ factors though a state of $C^*(E)$ if and only if $\epsilon$ is supported on the sources \cite[Corollary~6.1]{aHLRS1}. 

The hypothesis ``$\Lambda$ has no sources'' is frequently imposed because the Cuntz-Krieger relations of Kumjian and Pask \cite{KP} need to be substantially altered when there are sources. Exactly how they need to be altered was carefully analysed in \cite{RSY2}. However, because that paper was primarily about infinite graphs and problems that arise when vertices receive infinitely many edges, the answer is complicated. There is a family of ``locally convex'' row-finite $k$-graphs for which only a relatively simple adjustment is required \cite{RSY1}, but that will not suffice here. 

We again want to work with  a $k$-graph $\Lambda$ which is finite and has no sources or sinks. But the graph $\Lambda \backslash D$ appearing in Theorem~\ref{thm:case2-modified} can have sources. Indeed, the graph $\Lambda\backslash D$ whose skeleton  is drawn in Figure~\ref{sourceex2} below is not even locally convex. For the Toeplitz algebra, where the contentious Cuntz-Krieger relations are not imposed, we can safely apply \cite[Theorem~6.1]{aHLRS2} to graphs with sources because the ``without sources'' hypothesis was not used in its proof. However, to understand what happens on $C^*(\Lambda\backslash D)$, we have to deal with the Cuntz-Krieger relation of \cite{RSY2} for graphs with sources. Before we start, we pause to make clear what we mean by ``source''.

\begin{rmk}\label{rmk-absolute-sources}
In their orginal paper \cite{KP}, Kumjian and Pask say that a $k$-graph $\Lambda$ \emph{has no sources} if $v\Lambda^n$ is nonempty for every $n\in \N^k$. Thus they implicitly defined a source as a vertex for which at least one of the sets $v\Lambda^{e_i}$ is empty. This was the definition used in \cite{RSY1} and \cite{RSY2} (again implicitly, as suggested by the comments at the start of \cite[\S3]{RSY1}). In our analysis of KMS states on $C^*(\Lambda)$, the important sources $v$ are those that receive no edges at all, so that $v\Lambda=\{v\}$. To avoid possible confusion, we say for emphasis that $v$ is an \emph{absolute source} if $v\Lambda=\{v\}$. In the graph whose skeleton is drawn in Figure~\ref{sourceex2}, for example, both vertices $v$ and $w$ are sources in the sense of Kumjian and Pask, but only $w$ is an absolute source.
\end{rmk}

Our analogue of \cite[Corollary~6.1]{aHLRS1} for $k$-graphs is the following. Recall that when $\beta r_i>\ln\rho(A_i)$ for $1\leq i\leq k$,  by \cite[Theorem~6.1]{aHLRS2} the KMS$_\beta$ states on $\TC^*(\Lambda)$ for the  dynamics $\alpha^r$ have the form $\phi_\epsilon$ for $\epsilon$ belonging to a simplex $\Sigma_\beta$ in $[0,\infty)^{\Lambda^0}$.

\begin{prop}\label{KMSsources}
Suppose that $\Lambda$ is a finite $k$-graph, possibly with sources. Let $r\in (0,\infty)^k$ and consider $\beta$ such that $\beta r_i>\ln\rho(A_i)$ for $1\leq i\leq k$. Then a KMS$_\beta$ state $\phi_\epsilon$ of $(\TC^*(\Lambda), \alpha)$ factors through a state of $C^*(\Lambda)$ if and only if $\epsilon$ is supported on the set of absolute sources.
\end{prop}

For the proof of this result we recall some background material from \cite{RSY2}.

We need to work with the Cuntz-Krieger relations for $\Lambda$. We write $\Lambda^1$ for the set $\bigcup_{i=1}^k\Lambda^{e_i}$ of edges. If $v\in \Lambda^0$, a subset $E\subset v\Lambda^1$ is \emph{exhaustive} if every $\lambda\in v\Lambda$ with $d(\lambda)\not=0$ has $\Lambda^{\min}(\lambda,e)\not=\emptyset$ for some $e\in E$. (Equivalently, $E$ is exhaustive if for every $\lambda\in v\Lambda$, there exist an extension $\lambda\lambda'$ and $i\in \{1,2,\dots,k\}$ such that $(\lambda\lambda')(0,e_i)\in E$.) Then a Toeplitz-Cuntz-Krieger family $\{Q_v,T_\lambda\}$ is a Cuntz-Krieger family if
\begin{equation}\label{CKrel}
\prod_{e\in E}(Q_v-T_eT_e^*)=0\quad\text{for all $v\in \Lambda^0$ and all nonempty exhaustive $E\subset v\Lambda^1$.}
\end{equation}
The graph algebra or Cuntz-Krieger algebra $C^*(\Lambda)$ is the quotient of $\TC^*(E)$ by the ideal generated by the projections $\prod_{e\in E}(q_v-t_et_e^*)$ associated to the finite exhaustive subsets $E$ of $v\Lambda^1$.

\begin{rmk}
We have made some simplifications to the original Definition~2.5 of \cite{RSY2}. First, we have used Theorem~C.1 of \cite{RSY2} to restrict attention to exhaustive subsets of $\Lambda^1$. Second, we have observed that since our graph is finite, all exhaustive subsets of $\Lambda^1$ are finite, as required in part (iv) of \cite[Theorem~C.1]{RSY2}. And third, we have inserted the word ``nonempty'' to stress that we are not imposing any relation on $Q_v$ when $v\Lambda^1$ has no nonempty exhaustive subsets, which happens precisely when $v$ is an absolute source.
\end{rmk}

If $E\subset v\Lambda^1$ is an exhaustive set, then any larger subset $F\subset v\Lambda^1$ is also exhaustive, and the Cuntz-Krieger relation \eqref{CKrel} for $E$ implies the Cuntz-Krieger relation for $F$ (the product just has more terms). So it is of interest to find small exhaustive sets. When the graph $\Lambda$ has no sources, the subsets $v\Lambda^{e_i}$ are all exhaustive; since the $\{T_eT_e^*:e\in v\Lambda^{e_i}\}$ are mutually orthogonal projections, the product in the Cuntz-Krieger relation for $v\Lambda^{e_i}$ collapses, and we recover the usual Cuntz-Krieger relation
\[
0=\prod_{e\in v\Lambda^{e_i}}(Q_v-T_eT_e^*)=Q_v-\sum_{e\in v\Lambda^{e_i}}T_eT_e^*.
\]
Conversely, if $\Lambda$ has no sources and $\{Q_v,T_\lambda\}$ is a Cuntz-Krieger family in the usual sense, then it satisfies \eqref{CKrel} \cite[Lemma~B.3]{RSY2}. So in the absence of sources, we can find quite small exhaustive sets. When we add sources the exhaustive sets get bigger, often quite quickly. 

We now recall the construction of the state $\phi_\epsilon$ in \cite[Theorem~6.1]{aHLRS2}. It is defined spatially using the finite-path representation $\pi_{Q,T}$ on $\ell^2(\Lambda)$ and the usual orthonormal basis $\{h_\lambda\}$ for $\ell^2(\Lambda)$. We define weights 
\[
\Delta_\lambda=e^{-\beta r\cdot d(\lambda)}\epsilon_{s(\lambda)}, 
\]
and then 
\begin{equation*}
\phi_\epsilon(a)=\sum_{\lambda\in\Lambda}\Delta_\lambda\big(\pi_{Q,T}(a)h_\lambda\,|\,h_\lambda\big)\quad\text{for $a\in \TC^*(\Lambda)$.}
\end{equation*}
It will be helpful for us to know that for $\mu\in \Lambda$, $T_\mu T_\mu^*$ is the orthogonal projection on 
\begin{equation*}
T_\mu T_\mu^*(\ell^2(\Lambda))=\clsp\{h_{\lambda}:\lambda=\mu\nu\text{ for some $\nu\in s(\mu)\Lambda$}\},
\end{equation*}
and hence
\begin{equation*}
\phi_\epsilon(t_\mu t_\mu^*)=\sum_{\nu\in s(\mu)\Lambda}\Delta_{\mu\nu}.
\end{equation*}

\begin{proof}[Proof of Proposition~\ref{KMSsources}]
First we suppose that $\epsilon$ is not supported on the absolute sources, so that there is a vertex $v$ such that $v\Lambda\not=\{v\}$ and $\epsilon_v>0$. Then the basis vector $h_v$ belongs to the range of $Q_v=\pi_{Q,T}(q_v)$, and hence $Q_vh_v=h_v$. Since $v$ receives edges, the exhaustive subsets of $v\Lambda^1$ are all nonempty. For each such set $E$, and each $e\in E$, we have $T_eT_e^*h_v=0$, and hence 
\[
\pi_{Q,T}(q_v-t_et_e^*)h_v=(Q_v-T_eT_e^*)h_v=h_v\quad\text{for all $e\in E$.}
\]
Since the operators $Q_v-T_eT_e^*$ are all projections and commute with each other, all the summands in the formula for $\phi\big(\prod_{e\in E}(q_v-t_et_e^*)\big)$ are nonnegative, and
\[
\phi_\epsilon\Big(\prod_{e\in E}(q_v-t_et_e^*)\Big)\geq \Delta_v\Big(\prod_{e\in E}(Q_v-T_eT_e^*)h_v\;\Big|\;h_v\Big)=\Delta_v=\epsilon_v>0.
\]
Thus $\phi_\epsilon$ does not vanish on the kernel of the quotient map $q:\TC^*(\Lambda)\to C^*(\Lambda)$, and does not factor through a state of $C^*(\Lambda)$. 

For the converse, we suppose that $\epsilon$ is supported on the absolute sources, and aim to prove that $\phi_{\epsilon}$ vanishes on the projections appearing in \eqref{CKrel}. Since $\epsilon$ is then a convex combination of point masses at these sources, and since $\epsilon\mapsto \phi_\epsilon$ is affine, it suffices to prove this when $\epsilon$ is supported on a single absolute source $u$. We aim to show that $\phi_\epsilon$ vanishes on the products in the Cuntz-Krieger relation \eqref{CKrel} for all vertices $v$ which are not absolute sources. So we take such a vertex $v$, and a nonempty  exhaustive set $E\subset v\Lambda^1$.

We begin by rewriting the product $\prod(q_v-t_et_e^*)$ in \eqref{CKrel}, using that the projections $\{t_et_e^*:e\in \Lambda^{e_i}\}$ are mutually orthogonal. With $I_E:=\{i\in \{1,\dots,k\}:E\cap \Lambda^{e_i}\not=\emptyset\}$, we have
\begin{align*}
\prod_{e\in E}(q_v-t_et_e^*)&=\prod_{i\in I_E}\Big(\prod_{e\in E\cap \Lambda^{e_i}}(q_v-t_et_e^*)\Big)\notag\\
&=\prod_{i\in I_E}\Big(q_v-\sum_{e\in E\cap \Lambda^{e_i}}t_et_e^*\Big)\notag\\
&=q_v+\sum_{\emptyset\not=J\subset I_E}(-1)^{|J|}\prod_{i\in J}\Big(\sum_{e\in E\cap \Lambda^{e_i}}t_et_e^*\Big).
\end{align*}

Next we calculate the values of $\phi_\epsilon$ on the summands above. Since $\epsilon$ is supported on $u$, the weights $\Delta_\lambda$ vanish unless $s(\lambda)=u$, and
\[
\phi_\epsilon(q_v)=\sum_{\lambda\in v\Lambda u}\Delta_\lambda.
\]
For $J$ such that $\emptyset\not= J\subset I_E$, we have
\begin{align}
\phi_{\epsilon}\Big(\prod_{i\in J}&\Big(\sum_{e\in E\cap \Lambda^{e_i}}t_et_e^*\Big)\Big)\notag\\
&=\sum\Big\{\Delta_\lambda:\lambda\in v\Lambda u\text{ and }\sum_{e\in E\cap \Lambda^{e_i}}T_eT_e^*h_\lambda=h_\lambda\text{ for all $i\in J$}\Big\}\notag\\
&=\sum_{\{\lambda\in v\Lambda u\,:\,\lambda(0,e_i)\,\in\, E\text{ for all $i\in J$}\}}\Delta_\lambda.\label{phionprod}
\end{align}

We now consider a fixed $\lambda\in v\Lambda u$, and compute the coefficient of $\Delta_\lambda$ in 
\begin{equation}\label{multiplyout2}
\phi\Big(\prod_{e\in E}(q_v-t_et_e^*)\Big)=
\phi(q_v)+\sum_{\emptyset\not=J\subset I_E}(-1)^{|J|}\phi\Big(\prod_{i\in J}\Big(\sum_{e\in E\cap \Lambda^{e_i}}t_et_e^*\Big)\Big).
\end{equation} 
Since $v$ is not an absolute source and $u$ is, we have $d(\lambda)\not=0$. Then  $\Delta_\lambda$ occurs in the sum \eqref{phionprod} if and only if 
\[
J\subset I_{E,\lambda}:=\{i:\lambda(0,e_i)\in E\},
\]
and in that case has coefficient $(-1)^{|J|}$ in \eqref{multiplyout2}. Since $s(\lambda)=u$ is an absolute source, $\lambda$ has no extensions other than itself; since $E$ is exhaustive, this means that at least one initial edge $\lambda(0,e_i)$ must belong to $E$. Thus $I_{E,\lambda}$ is nonempty, $|I_{E,\lambda}|\geq 1$, and the coefficient of $\Delta_\lambda$ in \eqref{multiplyout2} is the number 
\[
1+\sum_{j=1}^{|I_{E,\lambda}|} (-1)^j\binom{|I_{E,\lambda}|}{j}=(1-1)^{|I_{E,\lambda}|}=0.
\]
We deduce that 
\[
\phi_{\epsilon}\Big(\prod_{e\in E}(q_v-t_et_e^*)\Big)=0.
\]

Finally, since the projections 
\[
\Big\{\prod_{e\in E}(q_v-t_et_e^*):\text{$v\in \Lambda^0$ is not an absolute source and $E\subset v\Lambda^1$ is exhaustive}\Big\}
\]
generate the kernel of the quotient map $q:\TC^*(\Lambda)\to C^*(\Lambda)$ and are fixed by the action $\alpha^r$, it follows from \cite[Lemma~2.2]{aHLRS1} that $\phi_{\epsilon}$ factors through a state of $C^*(\Lambda)$.
\end{proof}

In the examples we have been studying so far, the graphs $\Lambda$ have all had neither sinks nor sources in the strong sense of Kumjian and Pask \cite{KP}, and in these examples the same is true of the graphs $\Lambda\backslash H$ obtained by deleting a hereditary set $H$. 
In other graphs, it is possible that $\Lambda\backslash H$ has sources even if $\Lambda$ doesn't. The presence of sources may affect the KMS states. 

\begin{example}\label{creatingsources}
We consider a $2$-graph with skeleton shown in Figure~\ref{sourceex}. We can verify that there are $2$-graphs with this skeleton either by writing down the vertex matrices and checking they commute, or by checking that we have the same number of red-blue and blue-red paths between each pair of vertices.
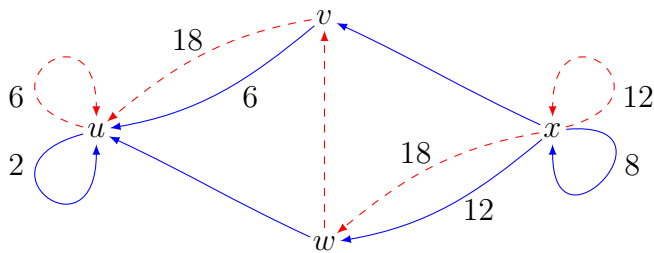
\begin{figure}[h]
\begin{tikzpicture}[scale=1.5]
 \node[inner sep=0.5pt, circle] (u) at (0,0) {$u$};
    \node[inner sep=0.5pt, circle] (v) at (2,1) {$v$};
    \node[inner sep=0.5pt, circle] (w) at (2,-1) {$w$};
    \node[inner sep=0.5pt, circle] (x) at (4,0) {$x$};
\draw[-latex, blue] (u) edge [out=195, in=270, loop, min distance=30, looseness=2.5] (u);
\draw[-latex, red, dashed] (u) edge [out=165, in=90, loop, min distance=30, looseness=2.5] (u);
\draw[-latex, blue] (v) edge [out=220, in=10] (u);
\draw[-latex, red, dashed] (v) edge [out=190, in=40] (u) ;
\draw[-latex, blue] (w) edge [out=155, in=335] (u) ;
\draw[-latex, blue] (x) edge [out=155, in=335] (v);
\draw[-latex, blue] (x) edge [out=220, in=10] (w);
\draw[-latex, red, dashed] (x) edge [out=190, in=40] (w);
\draw[-latex, red, dashed] (w) edge [out=90, in=270] (v);
\draw[-latex, blue] (x) edge [out=365, in=270, loop, min distance=30, looseness=2.5] (x);
\draw[-latex, red, dashed] (x) edge [out=15, in=90, loop, min distance=30, looseness=2.5] (x);
\node at (4.75, 0.3) {\color{black} $12$};
\node at (4.7,-0.3) {\color{black} $8$};
\node at (2.8, -0.2) {\color{black} $18$};
\node at (3.35,-0.7) {\color{black} $12$};
\node at (-.7, 0.3) {\color{black} $6$};
\node at (-.7,-0.3) {\color{black} $2$};
\node at (.8, 0.8) {\color{black} $18$};
\node at (1.35,0.3) {\color{black} $6$};
\end{tikzpicture}
\caption{A $2$-graph in which subgraphs can have sources.}\label{sourceex}
\end{figure}
For this graph we have $\rho(A_1)=\rho(A_{{\{x\}},1})=8$ and $\rho(A_2)=12$, so the preferred dynamics $\alpha^r$ has \[r_1=\ln 8\quad\text{and}\quad r_2=\ln 12.\] Proposition~\ref{ricond} implies that $r_1$ and $r_2$ are rationally independent. 

To analyse the KMS$_1$ states of $(\TC^*(\Lambda),\alpha^r)$, we first apply Theorem~\ref{thm:case2-modified} with $D=\{x\}$. Proposition~\ref{evectors} gives us a  common unimodular eigenvector $b^{-1}z$ of $A_1$ and $A_2$, and Proposition~\ref{existKMS1-modified} gives us a KMS$_1$ state $\psi$ which satisfies
\[
m^\psi=\begin{pmatrix}\psi(q_u)\\ \psi(q_v)\\ \psi(q_w)\\ \psi(q_x)\end{pmatrix}
=b^{-1}z=\frac{1}{24}\begin{pmatrix}3\\1\\12\\8\end{pmatrix}.
\]
A computation shows that $A_im^\psi=\rho(A_i)m^\psi$ for $i=1,2$, and it follows from \cite[Proposition~4.1]{aHLRS2} that $\psi$ factors through a state of $C^*(\Lambda)$. 
Theorem~\ref{thm:case2-modified} says that every KMS$_1$ state of $(\TC^*(\Lambda),\alpha^r)$ is a convex combination of $\psi$ and a state $\phi\circ q_D$ lifted from a KMS$_1$ state $\phi$ of $(\TC^*(\Lambda\backslash D),\alpha^r)$. 

The graph $\Lambda\backslash D$ has the skeleton shown in Figure~\ref{sourceex2}. Notice that it has two sources $w$ and $v$. 
\begin{figure}[h]
\begin{tikzpicture}[scale=1.5]
 \node[inner sep=0.5pt, circle] (u) at (0,0) {$u$};
    \node[inner sep=0.5pt, circle] (v) at (2,1) {$v$};
    \node[inner sep=0.5pt, circle] (w) at (2,-1) {$w$};
\draw[-latex, blue] (u) edge [out=195, in=270, loop, min distance=30, looseness=2.5] (u);
\draw[-latex, red, dashed] (u) edge [out=165, in=90, loop, min distance=30, looseness=2.5] (u);
\draw[-latex, blue] (v) edge [out=220, in=10] (u);
\draw[-latex, red, dashed] (v) edge [out=190, in=40] (u) ;
\draw[-latex, blue] (w) edge [out=155, in=335] (u) ;
\draw[-latex, red, dashed] (w) edge [out=90, in=270] (v);
\node at (-.7, 0.3) {\color{black} $6$};
\node at (-.7,-0.3) {\color{black} $2$};
\node at (.8, 0.8) {\color{black} $18$};
\node at (1.35,0.3) {\color{black} $6$};
\end{tikzpicture}
\caption{The skeleton of the $2$-graph $\Lambda\backslash D=\Lambda\backslash\{x\}$}\label{sourceex2}
\end{figure}
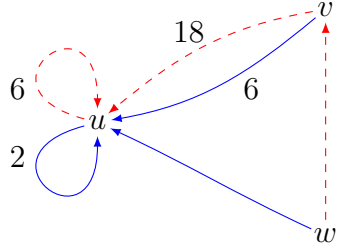
Since $r_1=\ln 8>\ln\rho(A_{{\Lambda^0\backslash D},1})=\ln 2$ and $r_2=\ln 12>\ln\rho(A_{{\Lambda^0\backslash D},2})=\ln 6$, we can apply \cite[Theorem~6.1]{aHLRS2} (recalling our previous observation that the proof of this theorem did not rely on the absence of sources) to get a $2$-dimensional simplex $\{\phi_\epsilon:\epsilon\in \Sigma_1\}$ of KMS$_1$ states of $(\TC^*(\Lambda\backslash D),\alpha^r)$.  Thus $(\TC^*(\Lambda),\alpha^r)$ has a $3$-dimensional simplex of KMS$_1$ states.

Proposition~\ref{KMSsources} implies that a KMS$_\beta$ state $\phi_{\epsilon}$ factors through a state of $C^*(\Lambda\backslash D)$ exactly when $\epsilon$ is supported on the absolute source $\{w\}$, and there is exactly one such KMS$_1$ state. Notice that $\phi_\epsilon\circ q_D$ cannot factor through a state of $C^*(\Lambda)$. Indeed, $C^*(\Lambda\backslash D)$ is not a quotient of $C^*(\Lambda)$: if an ideal $I$ in $C^*(\Lambda)$ contains $q_x$, then it also contains the projections $q_w$ and $q_v$, because the set $H=\{y\in \Lambda^0:q_y\in I\}$ is saturated as well as hereditary. (See \cite[\S5]{RSY1}, for example.)

 When $\beta>(\ln 12)^{-1}\ln 6$, so that  $\beta r_i>\ln \rho(A_{{\Lambda^0\backslash D},i})$ for $i=1$ and $2$, 
Theorem~6.1 of \cite{aHLRS2}  continues to apply  and gives a $2$-dimensional simplex of KMS$_\beta$ states of $(\TC^*(\Lambda\backslash D),\alpha^r)$. Taking limits as $\beta$ decreases to  $(\ln 12)^{-1}\ln 6$ gives KMS$_{(\ln 12)^{-1}\ln 6}$ states of of $(\TC^*(\Lambda\backslash D),\alpha^r)$.

Let $\phi$ be a  KMS$_{(\ln 12)^{-1}\ln 6}$ state of  $(\TC^*(\Lambda\backslash D),\alpha^r)$. We claim that $\phi$ vanishes on on $q_v$ and $q_w$. Using the Toeplitz-Cuntz-Krieger relation at $u$ with $n=e_2$ we get
\begin{align*}
\phi(q_u)&\geq \sum_{\lambda \in u\Lambda^{e_2}}\phi(t_\lambda t_\lambda^*)
\geq \sum_{\lambda \in u\Lambda^{e_2}}e^{-\beta r_2}\phi(q_{s(\lambda)})\\
&=6e^{-\beta r_2}\phi(q_u)+18e^{-\beta r_2}\phi(q_v)=\psi(q_u)+3\phi(q_v), 
\end{align*}
and hence  $\phi(q_v)=0$. Similarly 
\[
0=\phi(q_v)\geq  \sum_{\lambda \in v\Lambda^{e_2}}\phi(q_{s(\lambda)})=e^{-\beta r_2}\phi(q_w)
\]
implies that $\phi(q_w)=0$. 
By \cite[Lemma~2.2]{aHLRS1}, $\phi$ factors through a KMS$_{(\ln 12)^{-1}\ln 6}$ state of the quotient $\TC^*(\Lambda_{\{u\}})$. It follows from Propositions~4.2(c) of \cite{aHKR} that there is exactly one KMS$_{(\ln 12)^{-1}\ln 6}$ state of ($\TC^*(\Lambda_{\{u\}}), \alpha^r)$, and  that it does not factor through a state of $C^*(\Lambda_{\{u\}})$. Thus there is a unique KMS$_{(\ln 12)^{-1}\ln 6}$ state of  $(\TC^*(\Lambda\backslash D),\alpha^r)$.

\end{example}

Whether a given absolute source in a subgraph $\Lambda\backslash D$ gives rise to KMS$_1$ states depends on where the source is located relative to the critical components.

\begin{example} We consider a $2$-graph with skeleton shown in Figure~\ref{sourceex3}. 
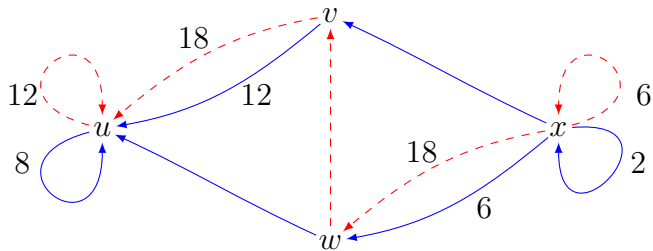
\begin{figure}[h]
\begin{tikzpicture}[scale=1.5]
 \node[inner sep=0.5pt, circle] (u) at (0,0) {$u$};
    \node[inner sep=0.5pt, circle] (v) at (2,1) {$v$};
    \node[inner sep=0.5pt, circle] (w) at (2,-1) {$w$};
    \node[inner sep=0.5pt, circle] (x) at (4,0) {$x$};
\draw[-latex, blue] (u) edge [out=195, in=270, loop, min distance=30, looseness=2.5] (u);
\draw[-latex, red, dashed] (u) edge [out=165, in=90, loop, min distance=30, looseness=2.5] (u);
\draw[-latex, blue] (v) edge [out=220, in=10] (u);
\draw[-latex, red, dashed] (v) edge [out=190, in=40] (u) ;
\draw[-latex, blue] (w) edge [out=155, in=335] (u) ;
\draw[-latex, blue] (x) edge [out=155, in=335] (v);
\draw[-latex, blue] (x) edge [out=220, in=10] (w);
\draw[-latex, red, dashed] (x) edge [out=190, in=40] (w);
\draw[-latex, red, dashed] (w) edge [out=90, in=270] (v);
\draw[-latex, blue] (x) edge [out=365, in=270, loop, min distance=30, looseness=2.5] (x);
\draw[-latex, red, dashed] (x) edge [out=15, in=90, loop, min distance=30, looseness=2.5] (x);
\node at (4.75, 0.3) {\color{black} $6$};
\node at (4.7,-0.3) {\color{black} $2$};
\node at (2.8, -0.2) {\color{black} $18$};
\node at (3.35,-0.7) {\color{black} $6$};
\node at (-.7, 0.3) {\color{black} $12$};
\node at (-.7,-0.3) {\color{black} $8$};
\node at (.8, 0.8) {\color{black} $18$};
\node at (1.35,0.3) {\color{black} $12$};
\end{tikzpicture}
\caption{Another $2$-graph in which subgraphs can have sources.}\label{sourceex3}
\end{figure}
Let $C=\{u\}$.
Then  $\rho(A_i)=\rho(A_{C,i})$ for all $i$,  and $C$ is the only $i$-critical component for each $i$. Since the hereditary closure of $C$ is all of $\Lambda^0$, Theorem~\ref{crit} implies that for the preferred dynamics $\alpha^r$, every  KMS$_1$ state of $(\TC^*(\Lambda),\alpha^r)$ factors through a state of $(\TC^*(\Lambda_C),\alpha^r)$.  With $D=\{x\}$, the graph $\Lambda\backslash D$ has an absolute source $w$ in $\Lambda\backslash D$. But this absolute source does not give a KMS state.
\end{example}

\appendix

\section{Rational independence}

We recall that real numbers $\{x_i:1\leq i\leq n\}$ are \emph{rationally independent} if 
\[
\sum_{i=1}^nc_ix_i=0\text{ and }c_i\in \Z\Longrightarrow c_i=0\text{ for $1\leq i\leq n$.}\]
In our situation, we find the hypothesis ``Suppose that $\{\ln \rho(A_i):1\leq i\leq k\}$ are rationally independent'' in several key results. See, for example, \cite[Proposition~3.1(b)]{aHLRS2}, \cite[Theorem~7.2]{aHLRS2} and \cite[Proposition~4.2(c)]{aHKR}.
When we are dealing with the preferred dynamics on the algebras of a $2$-graph with a single vertex, we have $r_i=\ln m_i:=\ln|\Lambda^{e_i}|$, and this condition simplifies to saying that $\ln m_1/\ln m_2$ is irrational. If so, the $k$-graph is aperiodic in the sense of Kumjian and Pask, and $C^*(\Lambda)$ is simple. The converse is not true, and is discussed in detail in \cite[\S3]{DY}. But it is often easy to decide rational independence, and hence settle the issue of aperiodicity, using the following simple number-theoretic characterisation of independence.

\begin{prop}\label{ricond}
Suppose that $m$ and $n$ are positive integers. Then $\ln m$ and $\ln n$ are rationally dependent if and only if there are positive integers $k$, $c$ and $d$ such that $c$ and $d$ are relatively prime and $m=k^c$, $n=k^d$.
\end{prop}

\begin{proof}
If $m=k^c$ and $n=k^d$, then $\ln m=c\ln k$, $\ln n =d\ln k$, and we have $d\ln m-c\ln n=0$. So we suppose that $\ln m$ and $\ln n$ are rationally dependent, say $a\ln m= b\ln n$ for $a,b\in \Z$. Then $m^a=n^b$. We write $m=\prod_p p^{m_p}$ and $n=\prod_p p^{n_p}$ for the prime factorisations of $m$ and $n$. Then $\prod_p p^{am_p}=\prod_p p^{bn_p}$, and uniqueness of the prime factorisation implies that $am_p=bn_p$ for all $p$. We write $d=\gcd(a,b)^{-1}a$ and $c=\gcd(a,b)^{-1}b$, and we still have $dm_p=cn_p$ for all $p$. Since $c$ and $d$ are relatively prime, we deduce that for all $p$, $d$ divides $n_p$ and $c$ divides $m_p$. Say $m_p=ck_p$ and $n_p=d l_p$. Then we have
\[
d(ck_p)=dm_p=cn_p=c(dl_p)\text{ for all $p$.}
\]
Since $c$ and $d$ are positive, we deduce that $k_p=l_p$ for all $p$. Now take $k:=\prod_p p^{k_p}$, and we have
\[
\textstyle{k^c=\prod_p p^{ck_p}=\prod_p p^{m_p}=m\quad\text{and}\quad k^d=\prod_p p^{dk_p}=\prod_p p^{dl_p}=\prod_p p^{n_p}=n.}\qedhere
\]
\end{proof}


\begin{thebibliography}{19}


\bibitem{BR} O. Bratteli and D.W. Robinson, Operator Algebras and Quantum Statistical Mechanics II, second edition, Springer-Verlag, Berlin, 1997.


\bibitem{CL} T.M. Carlsen and N.S. Larsen, \emph{Partial actions and KMS states on relative graph $C^*$-algebras}, J. Funct. Anal. \textbf{271} (2016), 2090--2132.

\bibitem{C} C. Chlebovec, \emph{KMS states for quasi-free actions on finite-graph algebras}, J. Operator Theory \textbf{75} (2016), 119--138.

\bibitem{CT} J. Christensen and K. Thomsen, \emph{Finite digraphs and KMS states}, J. Math. Anal. Appl. \textbf{433} (2016), 1626--1646.

\bibitem{DY} K.R. Davidson  and D. Yang, \emph{Periodicity in rank 2 graph
algebras},  Canad. J. Math. \textbf{61} (2009), 1239--1261.

\bibitem{EFW} M. Enomoto, M. Fujii and Y. Watatani, \emph{KMS states for gauge action on $\OO_A$}, Math. Japon. \textbf{29} (1984), 607--619.

\bibitem{EL} R. Exel and M. Laca, \emph{Partial dynamical systems and the KMS condition}, Comm. Math. Phys. \textbf{232} (2003), 223--277.

\bibitem{FR} N.J. Fowler  and I. Raeburn, \emph{The Toeplitz algebra of a Hilbert 
bimodule}, Indiana Univ. Math. J. \textbf{48} (1999), 155--181.

\bibitem{HRSW} R. Hazlewood, I. Raeburn, A. Sims and S.B.G. Webster, \emph{Remarks on some fundamental results about higher-rank graphs and their $C^*$-algebras}, Proc. Edinburgh Math. Soc. \textbf{56} (2013), 575--597.

\bibitem{aHKR} {A. an Huef, S. Kang and I. Raeburn}, \emph{Spatial realisations of KMS states on the $C^*$-algebras of higher-rank graphs}, J. Math. Anal. Appl. {\bf 427} (2015), 977--1003.

\bibitem{aHLRS1} {A. an Huef, M. Laca, I. Raeburn and A. Sims}, \emph{KMS states on the $C^*$-algebras of finite graphs}, J. Math. Anal. Appl. {\bf 405} (2013), 388--399.

\bibitem{aHLRS2} {A. an Huef, M. Laca, I. Raeburn and A. Sims}, \emph{KMS states on $C^*$-algebras associated to higher-rank graphs}, J. Funct. Anal.  {\bf 266} (2014), 265--283.

\bibitem{aHLRS3} {A. an Huef, M. Laca, I. Raeburn and A. Sims}, \emph{KMS states on the $C^*$-algebra of a higher-rank graph and periodicity in the path space}, J. Funct. Anal. \textbf{268} (2015), 1840--1875.

\bibitem{aHLRS4} {A. an Huef, M. Laca, I. Raeburn and A. Sims}, \emph{KMS states on the $C^*$-algebras of reducible graphs}, Ergodic Theory Dynam. Systems \textbf{35} (2015), 2535--2558.

\bibitem{aHRabel} A. an Huef and I. Raeburn, \emph{Equilibrium states of graph algebras}, in Operator Algebras and Applications, Proc. 2015 Abel Symposium, Springer, 2016, pp. 171--183.

\bibitem{IK} M. Ionescu and A. Kumjian, \emph{Hausdorff measures and KMS states}, Indiana Univ. Math. J. \textbf{62} (2013), 443--463.

\bibitem{KW} T. Kajiwara and Y. Watatani, \emph{KMS states on finite-graph $C^*$-algebras}, Kyushu J. Math. \textbf{67} (2013), 83--104.




\bibitem{KP}{A. Kumjian and D. Pask}, \emph{Higher-rank graph $C^*$-algebras}, New York J. Math. {\bf 6} (2000), 1--20.



\bibitem{LR} M. Laca and I. Raeburn, \emph{Phase transition on the Toeplitz algebra of the affine semigroup over the natural numbers}, Adv. Math. \textbf{225} (2010), 643--688.

\bibitem{LRR} M. Laca, I. Raeburn and J. Ramagge, \emph{Phase transitions on Exel crossed products associated to dilation matrices}, J. Funct. Anal. \textbf{261} (2011), 3633--3664.

\bibitem{M} R. McNamara, \emph{KMS states of graph algebras with generalised gauge dynamics}, PhD thesis, Univ. of Otago, 2015.


\bibitem{R} I. Raeburn, Graph Algebras, CBMS Regional Conference  Series in Math., vol.~103, Amer. Math. Soc., Providence, 2005.

\bibitem{RS}{I. Raeburn and A. Sims}, \emph{Product systems of graphs and the Toeplitz algebras of higher-rank graphs}, J. Operator Theory {\bf 53} (2005), 399--429.

\bibitem{RSY1}{I. Raeburn, A. Sims and T. Yeend}, \emph{Higher-rank graphs and their $C^*$-algebras}, Proc. Edinburgh Math. Soc. \textbf{46} (2003), 99--115.

\bibitem{RSY2}{I. Raeburn, A. Sims and T. Yeend}, \emph{The $C^*$-algebras of finitely aligned higher-rank graphs}, J. Funct. Anal. {\bf 213} (2004),  206--240.





\bibitem{Sen}{E. Seneta}, Non-Negative Matrices and Markov Chains, second edition, Springer-Verlag, New York, 1981.


\bibitem{S2}{A. Sims}, \emph{Gauge-invariant ideals in the $C^*$-algebras of finitely aligned higher-rank graphs}, Canad. J. Math. \textbf{58} (2006), 1268--1290.

\bibitem{SWW} A. Sims, B. Whitehead and M.F. Whittaker, \emph{Twisted $C^*$-algebras associated to finitely aligned higher-rank graphs}, Documenta Math. {\bf 19} (2014), 831--866.

\bibitem{T} K. Thomsen, \emph{KMS weights on groupoid and graph $C^*$-algebras}, J. Funct. Anal. \textbf{266} (2014), 2959--2988.



\end{thebibliography}
\end{document}